\documentclass{amsart} % TK
\usepackage{amsmath,amsthm,amssymb,bm,amsrefs}
\usepackage{tikz-cd}
\allowdisplaybreaks

\newtheorem{defin}{Definition}
\newtheorem{prop}{Proposition}
\newtheorem{theorem}{Theorem}
\newtheorem{lemma}{Lemma}
\newtheorem{cor}{Corollary}
\newtheorem{remark}{Remark}

\newtheorem{example}{Example}

\newcommand{\fn}{^-{}}
\newcommand{\sn}{^\thicksim {}}
\newcommand{\cpol}[1]{\boldsymbol{#1}}

\newcommand{\lres}{\mathop{\backslash}}
\let\ld=\lres
\newcommand{\rres}{\mathop{/}}
\let\rd=\rres
\newcommand{\var}[1]{\mathsf{#1}}     % named varieties
\newcommand{\avar}[1]{\mathcal{#1}}    % arbitrary varieties
\newcommand{\dfeq}{\mathrel{\mathop:}=}  % definitional =

\newcommand{\pfpmv}{pf\Psi MV}    % perfect pseudo MV (category)
\newcommand{\ppmv}{P\Psi MV}      % perfectly generated pseudo MV (variety)
\newcommand{\cigmv}{CanIGMV}      % cancellative integral generalised MV

\newcommand{\pf}{\textit{pf}}

\tikzstyle{vertex}=[circle,fill=blue!30,draw=blue,minimum size=1.5mm,inner
sep=1pt,label distance=0pt]
\tikzstyle{every label}=[label distance=0pt]
\tikzstyle{every edge}=[draw=black,thick]
\tikzstyle{dot}=[circle,fill,draw,minimum size=1.5mm,inner sep=0pt]
\tikzstyle{odot}=[circle,draw,minimum size=1.5mm,inner sep=0pt]

\def\pprod{\mathop{\mathchoice{\raisebox{-1ex}{\scalebox{2}[2]{{\rm X}\kern-1.5pt}}}%
{\raisebox{-2.5pt}{\scalebox{1.2}[1.45]{{\rm X}}}}%
{\raisebox{-1.6pt}{{\rm X}}}%
{\raisebox{-1.1pt}{\scalebox{0.6}[0.7]{{\rm X}}}}}}

\newcommand{\ord}{\mathop{\rm dim}}

\title[Kites and pseudo MV-algebras]{Kites and representations of pseudo MV-algebras}
\author[Botur]{Michal Botur$^\dag$}
\author[Kowalski]{Tomasz Kowalski$^\ddag$}

\address{$^\dag$ Palack\' y University Olomouc, Faculty of Sciences}
\email{michal.botur@upol.cz}

\address{$^\ddag$ Department of Logic, Jagiellonian University}
\email{tomasz.s.kowalski@uj.edu.pl}

\thanks{The first author acknowledges support
  by IGA, project P\v{r}F 2021 030 and 2022 17.}


\begin{document}

\begin{abstract}
We investigate the structure of perfect residuated lattices, focussing
especially on perfect pseudo MV-algebras. We show that perfect pseudo
MV-algebras can be represented as a generalised version of kites
of Dvure\v{c}enskij and Kowalski, and that they are categorically equivalent to
$\ell$-groups with a distinguished automorphism. 
We then characterise varieties generated by kites and describe
the lattice of these varieties as a complete sublattice of the lattice of
perfectly generated varieties of perfect pseudo MV-algebras.
\end{abstract}

\maketitle
\section{Introduction}\label{sec:intro}

This article grew out of an attempt at answering Question~8.4 from
Dvure\v{c}enskij, Kowalski~\cite{DK14}. 
The line of research begun there deals with variants of a rather special
construction of certain algebras called \emph{kites}. 
Kites are most naturally associated with
a noncommutative generalisation of BL-algebras known as \emph{pseudo
  BL-algebras} (see~\cite{DK14} and Dvure\v{c}enskij~\cite{Dvu15}),
however the construction has been used in a broader context
of residuated lattices (e.g., Botur, Dvure\v{c}enskij~\cite{BD18})
and algebras related to quantum computation (e.g., Dvure\v{c}enskij,
Holland~\cite{DH14}; Botur, Dvure\v{c}enskij~\cite{BD15} 
and Dvure\v{c}enskij~\cite{Dvu13}). Another root of the article is
Di Nola, Dvure\v{c}enskij, Tsinakis~\cite{DDT08}, where the notion of
a \emph{perfect MV-algebra} was first generalised to a wider class of residuated
lattices (note however that the terminology of~\cite{DDT08} differs from
what is now the established terminology). 

We begin by establishing a few facts about \emph{perfect residuated lattices},
a naturally defined class inspired by~\cite{DDT08}, but never studied 
in this generality. The class of perfect residuated lattices is a peculiar
mixture of a variety and antivariety, which deserves to be investigated further. 
This is the content of Section~\ref{sec:perfect-rls}. 
Later, we focus on \emph{pseudo MV-algebras} and for the
most part indeed on \emph{perfect pseudo MV-algebras}. This narrowing of view
bears fruit: we obtain several structural results that we believe would not be
possible to discover if we started looking through a wider lens. 

In Section~\ref{sec:perfect-psmv} we start working with perfect pseudo
MV-algebras. We define a very natural generalisation of the kite construction,
and prove that every perfect pseudo MV-algebra is isomorphic to one such, which
enables us to show that the category of pseudo MV-algebras
(where arrows are homomorphisms) is equivalent to the category of $\ell$-groups with a
distinguished automorphism, say $\lambda$
(where arrows are homomorphisms commuting with $\lambda$).
This generalises the results of Di Nola, Lettieri~\cite{DL94} and
Di Nola, Dvure\v{c}enskij, Tsinakis~\cite{DDT08} stating, respectively, that
perfect MV-algebras are categorically equivalent to Abelian $\ell$-groups, and
that \emph{perfect symmetric pseudo MV-algebras} (called perfect GMV-algebras
in~\cite{DDT08}) are categorically equivalent to $\ell$-groups. 
The category of $\ell$-groups with a
distinguished automorphism is itself categorically
equivalent to a variety in the signature extending $\ell$-groups by two unary
functions. 

Section~\ref{sec:kites-and-ppmv} brings back the old kites, but presents them
as hybrid structures consisting of and $\ell$-group and a bi-unary algebra,
which acts on a power of the $\ell$-group by permuting coordinates. 
Curiously, precisely these bi-unars were used in Baldwin, Berman~\cite{BB75}
as the  very first example of a variety with arbitrarily large finite
subdirectly irreducibles but no infinite ones. 

Section~\ref{sec:approximation} deals with the following problem.
It is shown in Section~\ref{sec:perfect-psmv} that every perfect pseudo MV-algebra
is a ``generalised kite''. A quick example shows that the same is not true for
kites. Now, how closely can one approximate a perfect pseudo MV-algebra by
kites? We show that for any given perfect pseudo MV-algebra $\mathbf{A}$ and
an $\ell$-group $\mathbf{L}$, there exists a kite over $\mathbf{L}$, such that
it is the best possible approximation of $\mathbf{A}$ by a kite over $\mathbf{L}$. 

In Section~\ref{sec:varieties} we look at the lattice of varieties generated by
kites as a subposet of the lattice of \emph{perfectly generated varieties} of
pseudo MV-algebras, that is, the varieties generated 
by their perfect members. We show that the former is a complete
sublattice of the latter. In particular,
every perfectly generated variety has a unique kite-generated interior, and a
unique kite-generated closure. The lattice of varieties generated by kites
decomposes as an ordinal sum of a singleton and the direct product
of nontrivial varieties of $\ell$-groups with the divisibility lattice.

\section{Preliminaries}\label{sec:prel}

For the universal algebraic background, the first few chapters of Burris,
Sankappanavar~\cite{BS81} or of~\cite{Ber12} will suffice. For category theory
generalities, Mac Lane~\cite{ML98} has more than enough, but for categorical
equivalences between varieties and related classes of algebras
McKenzie~\cite{McK94} will be useful. Generally we will
identify varieties with appropriate categories, or, to be precise, with classes of
objects of categories whose arrows are homomorphisms. 

We follow the universal algebraic tradition of using
boldface letters to denote algebras, and corresponding lightface for their
universes. Sometimes the notation will be to complex for an easy
boldface/lightface distinction; at such places we will use the categorical 
``absolute value'' notation instead.
For classes of algebras, especially for varieties, we adopt the
conventions from Galatos, Jipsen, Kowalski, Ono~\cite{GJKO07}, namely, we use
calligraphic letters for 
arbitrary classes of algebras, but for named classes we use their
acronyms in sans serif. So for example $\mathcal{V}\subseteq \var{MV}$ will
stand for an arbitrary variety of MV-algebras. 
The names of varieties of residuated lattices, and acronyms for them, we also take
from~\cite{GJKO07}, with the exception of the variety of pseudo MV-algebras
which we call $\var{\Psi MV}$, and not $\var{psMV}$ as in~\cite{GJKO07}.
If needed, other additions to and alterations of these conventions will be 
introduced as we go along. 

For any variety $\mathcal{V}$ we write $\Lambda(\mathcal{V})$ for the lattice of
subvarieties of $\mathcal{V}$, and $\Lambda^+(\mathcal{V})$ for 
the poset of nontrivial subvarieties of $\mathcal{V}$.

\subsection{Residuated lattices and FL-algebras}
We work within the general framework of \emph{residuated lattices},
that is, algebras
$\mathbf A=(A;\wedge, \vee,\cdot,\lres,\rres,1)$ such that
\begin{itemize}
\item[(RL1)] $(A;\wedge,\vee)$ is a lattice, 
\item[(RL2)] $(A;\cdot,1)$ is a monoid, 
\item[(RL3)] the equivalences
$$
y\leq x\lres z \quad\Leftrightarrow\quad 
xy\leq z \quad\Leftrightarrow\quad  x\leq z\rres y
$$
hold for all $x,y,z\in A$.
\end{itemize}
The ordering relation $\leq$ is the natural lattice order on $A$,
and multiplication is written as juxtaposition, unless there is a good reason
not to write it so.

The \emph{residuation equivalences} in (RL3) are themselves equivalent to
four identities, so residuated lattices form a variety, denoted $\var{RL}$.
Expansions of residuated lattices by an additional constant $0$ are known
as \emph{FL-algebras} (for \emph{\textbf{F}ull \textbf{L}ambek calculus});
they form a variety $\var{FL}$, defined by the same identities as residuated
lattices, so $0$ is not assumed to have any special properties, and it is there
only to make it possible to define \emph{negations}, that is, the 
operations: 
$$
x\fn \dfeq  0\rres x \quad\text{and}\quad x\sn \dfeq   x\lres 0,
$$
which will play an important role in the article. FL-algebras satisfying
$x\fn = x\sn$ are called \emph{symmetric}. 
In the signature expanded by $0$, residuated lattices are term equivalent to
FL-algebras satisfying the identity $0=1$. The binding strength of the
operations is, from strongest to weakest: negations, multiplication, division, lattice
operations. Occasionally we may add unnecessary parentheses for legibility. 

Now we briefly recall some facts about
residuated lattices and FL-algebras that we will use. 
For more we refer the reader to~\cite{GJKO07}; we follow the
terminology and notation from there as closely as practicable.
The core of the arithmetic of residuated lattices is built
upon the following identities:
\begin{align*}
\bigvee X \cdot \bigvee Y &= \bigvee\{xy: x\in X, y\in Y\},\\
\bigl(\bigvee X\bigr) \lres z &= \bigwedge\{x\lres z: x\in X\}, & 
z \rres\bigl(\bigvee Y\bigr) &= \bigwedge\{z\rres y: y\in Y\},\\
z \lres \bigl(\bigwedge Y\bigr) &= \bigwedge\{z\lres y: y\in Y\}, &
\bigl(\bigwedge X\bigr) \rres z &= \bigwedge\{x\rres z: x\in X\},                      
\end{align*}
where $X$ and $Y$ are any subsets of the universe $A$ of a residuated lattice
$\mathbf{A}$, and $z\in A$. All equalities are to be read conditionally, that is,
if the left-hand side exists, then so does the right-hand side and they are
equal. The following De Morgan laws follow immediately, if $0$ is in the
signature:
$$
\bigl(\bigvee X\bigr)\sn = \bigwedge\{x\sn: x\in X\},\qquad\qquad\quad
\bigl(\bigvee Y\bigr)\fn = \bigwedge\{y\fn: y\in Y\}.
$$
We will use these identities without further ado, as is the custom in the
trade.

An FL-algebra $\mathbf{A}$ is 
\emph{integral} if $1$ is the largest element of $A$; it
is \emph{0-bounded} if $0$ is the smallest element of $A$.
Integral 0-bounded FL-algebras are known as
FL$_w$-algebras ($w$ for \emph{weakening}), so
according to our conventions $\var{FL_w}$ will stand for the variety of
FL$_w$-algebras. 

Let $\mathbf{A}$ be an FL-algebra, and $a,b\in A$. The \emph{left
  conjugate} of $a\in A$ by $b\in A$ is the element
$\lambda_b(a) \dfeq  (b\lres ab) \wedge 1$ and
the \emph{right conjugate} is $\rho_b(a) \dfeq  (ba\rres b)\wedge 1$.
A subalgebra $\mathbf{S}$ of $\mathbf{A}$ is \emph{normal} in $\mathbf{A}$
if $S$ is closed under conjugation by all elements of $A$. It is
\emph{convex} if $S$ is convex as a subset of $(A;\leq)$.
A \emph{negative cone} $A$ is the set $A\!^-\dfeq \{x\in A: a\leq 1\}$.
Negative cones and convex normal subalgebras of $0$-free reducts of FL-algebras are
structurally important in view of the following fact. 

\begin{prop}\label{prop:CNS}
The congruence lattice of an FL-algebra $\mathbf{A}$ is isomorphic to the lattice of
convex normal subalgebras of the $0$-free reduct $\mathbf{A}\!^{rl}$ of
$\mathbf{A}$. Moreover, any congruence $\theta$ on $\mathbf{A}$ is determined by the
class $[1]_\theta\cap A^-$.
\end{prop}

For an arbitrary FL-algebra $\mathbf{A}$, its convex normal subalgebras
are not upward closed in general, but if $\mathbf{A}$ is integral then
they are, and then they coincide with \emph{normal filters}
(called deductive filters in~\cite{GJKO07}).
A normal filter is a set $F\subseteq A$ such that (i)~$F$ 
is a lattice filter with $1\in F$, (ii) $a,b\in F$ implies $ab\in F$, (iii)
$a\in F, b\in A$ implies $\lambda_b(a),\rho_b(a)\in F$.
In integral FL-algebras, the conjugates simplify to
$\lambda_b(a) = b\lres ab$ and $\rho_b(a) = ba\rres b$.

The next lemma can be easily derived from general structure theory of residuated
lattices, but since it is important for our results, we will present it
with a proof. It will also give the readers not familiar with residuated
lattices an opportunity to see some of the arithmetic of residuated lattices in action.
Note that in any residuated lattice $\mathbf{A}$, for any
$a,b,x,y\in A$, we have:  
\begin{enumerate}
\item $a\cdot \lambda_a(x)\leq x\cdot a,$ $\rho_a(x)\cdot a\leq a\cdot x$,
\item $ \lambda_a(x)\cdot \lambda_a(y)\leq \lambda_a (x\cdot y),$
  $\rho_a(x)\cdot \rho_a(y)\leq\rho_a (x\cdot y)$, 
\item $\lambda_b (\lambda_a (x))\leq \lambda_{a\cdot b}(x),$ $\rho_a (\rho_b
  (x))\leq \rho_{a\cdot b}(x)$. 
\end{enumerate}
A \emph{conjugation polynomial} $\cpol{\alpha}$ over $\mathbf{A}$ is any
unary polynomial 
$(\gamma_{a_1}\circ\gamma_{a_2}\circ\cdots\circ\gamma_{a_n})(x)$ 
where $\gamma\in\{\lambda,\rho\}$ and $a_i\in A$ for $1\leq i\leq n$. 
We write $\mathrm{cPol}(\mathbf{A})$ for the set of all conjugation
polynomials over $\mathbf{A}$. 
For an element $u\in A$, an \emph{iterated conjugate} of $u$ is
$\cpol{\alpha}(u)$ for some $\cpol{\alpha}\in \mathrm{cPol}(\mathbf{A})$.

\begin{lemma}\label{FILTER}
Let $\mathbf A$ be an integral FL-algebra.
Then for any normal filter $F\subseteq A$ and any $x\in A$ the set 
$$
F_x=\{a\in A\mid \exists f\in F\,\exists
\cpol{\alpha_1},\dots,\cpol{\alpha_n}\in \mathrm{cPol}(\mathbf{A}) 
\text{ such that } f \cdot\cpol{\alpha_1}(x)\cdots\cpol{\alpha_n}(x)\leq a\}
$$ 
is the smallest normal filter containing $F\cup\{x\}$.
\end{lemma}

\begin{proof}
As any normal filter is an up-set closed under products and iterated conjugates,
$F_x$ is clearly contained in the smallest normal filter
containing $F\cup\{x\}$. 

For converse it suffices to show that for any
$a,b\in F_x$ and $\cpol{\delta}\in\mathrm{cPol}(\mathbf{A})$
we have $\cpol{\delta}(ab)\in F_x$.
Let $a,b\in F_x$ and $\cpol{\delta}\in\mathrm{cPol}(\mathbf{A})$. Then there exist
$f_1,f_2\in F$ and
$\cpol{\alpha_1},\dots\cpol{\alpha_n},\cpol{\beta_1},\dots,\cpol{\beta_m}\in
\mathrm{cPol}(\mathbf{A})$ such that 
$$
a\geq f_1\cdot \cpol{\alpha_1}(x)\cdots \cpol{\alpha_n}(x)\text{ and }b\geq
f_2\cdot \cpol{\beta_1}(x)\cdots \cpol{\beta_m}(x).
$$
Then, using the properties (1)--(3) of conjugates, and
writing $\prod_{i=1}^n \cpol{\alpha_i}(x)$ for
the product $\cpol{\alpha_1}(x)\dots \cpol{\alpha_n}(x)$ and similarly for
$\cpol{\beta_j}$, we obtain
\begin{align*}
\cpol{\delta} (a\cdot b) &\geq  \cpol{\delta}\bigl(f_1\cdot
 \prod_{i=1}^n \cpol{\alpha_i}(x)\cdot f_2\cdot
 \prod_{j=1}^m \cpol{\beta_j}(x)\bigr)\\ 
                         &\geq  \cpol{\delta}(f_1)\cdot
\prod_{i=1}^n \cpol{\delta}(\cpol{\alpha_i}(x))\cdot\cpol{\delta}(f_2)\cdot
          \prod_{j=1}^m \cpol{\delta}(\cpol{\beta_j}(x)).
\end{align*}
Then letting $w = \prod_{i=1}^n \cpol{\delta}(\cpol{\alpha_i}(x))$
and $u = \prod_{j=1}^m \cpol{\delta}(\cpol{\beta_j}(x))$, we get 
$$
\cpol{\delta}(f_1)\cdot w\cdot\cpol{\delta}(f_2)\cdot u 
\geq  \cpol{\delta}(f_1) \cdot \rho_w(\cpol{\delta}(f_2))\cdot
      w\cdot u. 
$$
Since $F$ is a normal filter, and
$\cpol{\delta}\circ\cpol{\alpha_1},\dots,\cpol{\delta}\circ\cpol{\alpha_n},
\cpol{\delta}\circ\cpol{\beta_1},\dots,\cpol{\delta}\circ\cpol{\beta_m}
\in\mathrm{cPol}(\mathbf{A})$,
we have
$\cpol{\delta}(f_1) \cdot \rho_w(\cpol{\delta}(f_2))\cdot w\cdot u\in F$.
Hence  $\cpol{\delta} (ab)\in F_x$ as required.
\end{proof}

Since we will frequently jump between
the context of residuated lattices (without bounds) and FL$_w$-algebras (with
bounds) we introduce 
an \emph{ad-hoc} notational device. To indicate the residuated lattice context,
we will write $e$ instead of $1$ for the unit element. That is, whenever
we write $e$ we tacitly assume $e=1=0$. Otherwise, we write $1$, and
assume the  FL$_w$-algebra context. Note that this convention also means that
appearance of $1$ implies integrality, but appearance of $e$ does not.
Furthermore, following the tradition we will use `residuated lattice'
as a generic name, meaning residuated-lattice-or-FL-algebra. 

\subsection{Negative cones of $\ell$-groups, and pseudo MV-algebras} 
A \emph{lattice ordered group} ($\ell$-group) is an algebra $\mathbf
L=(L;\wedge,\vee,\cdot,{}^{-1}, e)$ where $(L;\wedge, \vee)$ is a lattice,
$(L;\cdot,{}^{-1},e)$ is a group and  
\begin{align*}
x(y\wedge z) w &= x y w\wedge x z w,\\
  x (y\vee z) w &= x y w\vee x z w
\end{align*}
hold for any $x,y,z,w\in L$. The variety $\var{LG}$ of $\ell$-groups is
term equivalent to a variety of residuated lattices defined by the
identity $x(x\ld e) = e$, by putting $x\lres y \dfeq x^{-1}y$ and 
$x\rres y \dfeq xy^{-1}$ one way, and $x^{-1} \dfeq x\lres e$ the other.
More background on $\ell$-groups can be found in~\cite{KM94}. For our purposes
here, it will suffice to recall that any $\ell$-group $\mathbf{L}$ is completely
determined by the residuation structure of its negative cone
$L^-=\{x\in L:x\leq e\}$. Namely, defining the algebra
$$
\mathbf L^- = (L^-;\wedge,\vee,\cdot,\ld,\rd,e) 
$$
where $e$, $\wedge$, $\vee$ and $\cdot$ are inherited from
$\mathbf{L}$, and
$$
x\rres y\dfeq xy^{-1}\wedge e, \qquad
y\lres x \dfeq  y^{-1}x\wedge e
$$
we obtain an integral residuated lattice satisfying the identities
\begin{align*}
xy\rres y &= x = y\lres yx\tag{\textrm{Can}}\label{eq:can}\\
x\rres (y\lres x) &= x\vee y =(x\rres y)\lres x.\tag{\textrm{\L{}uk}}\label{eq:luk}
\end{align*}
The first of these is equivalent over residuated lattices to the usual
cancellation laws
$$zx=zy \Rightarrow x=y\qquad\text{and}\qquad  
xz=yz  \Rightarrow x=y.$$
The second amounts to a non-commutative rendering of the
\emph{\L{}ukasiewicz axiom}
$(x\rightarrow y)\rightarrow y = (y\rightarrow x)\rightarrow x$.
Note that $\ell$-groups in the signature of residuated lattices
satisfy~\eqref{eq:can} but not~\eqref{eq:luk}, however they do satisfy
\begin{equation}
x\rres ((x\vee y)\lres x) = x\vee y =(x\rres (x\vee y))\lres
x,\tag{\textrm{Wuk}}\label{eq:wuk} 
\end{equation}
which in integral residuated lattices is equivalent to~\eqref{eq:luk}. 
Residuated lattices satisfying~\eqref{eq:can} are obviously known
as \emph{cancellative}. Commutative integral zero-bounded FL-algebras
satisfying~\eqref{eq:luk} are known as \emph{MV-algebras}. Residuated lattices
satisfying~\eqref{eq:wuk} were investigated in Galatos, Tsinakis~\cite{GT05}, under
the name of \emph{generalised MV-algebras (GMV-algebras)}. This is the standard
terminology now, but in the past (notably in Di Nola, Dvure\v{c}enskij,
Tsinakis~\cite{DDT08}) the name GMV-algebras was also used for what is now
known as \emph{pseudo MV-algebras} which we will define shortly. 
Integral GMV-algebras satisfying~\eqref{eq:wuk} turn out to be
precisely the integral residuated lattices satisfying~\eqref{eq:luk}; they are
called \emph{integral generalised MV-algebras (IGMV-algebras)}.
Cancellative residuated lattices can be integral (although 
they cannot be bounded below), so the identities~\eqref{eq:can}
and~\eqref{eq:luk} jointly define a variety $\var{\cigmv}$ of
\emph{cancellative integral generalised MV-algebras}, whose
members are precisely the negative cones of $\ell$-groups. Since any
residuated lattice homomorphism between $\ell$-groups restricts to a homomorphism
of (the algebras defined on) their negative cones, the map
$$
{}^-\colon \var {LG}\rightarrow \var
{\cigmv}
$$
is a functor. Moreover, it has a right adjoint 
$$
\ell\colon \var {\cigmv}\rightarrow \var {LG}
$$ 
such that $\ell(\mathbf A)^-=\mathbf A$ for any $\mathbf A\in \var
{\cigmv}$ and $\ell(\mathbf L^-)\cong \mathbf L$ for any $\mathbf L\in
\var {LG}$. These functors establish a categorical equivalence
between $\var{LG}$ and $\var{\cigmv}$. In particular, the subvariety lattices
of $\var{LG}$ and $\var{\cigmv}$ are isomorphic.

The next proposition will not be used until Section~\ref{sec:perfect-psmv}
(proof of Lemma~\ref{KiteHom}). Since it belongs to $\ell$-group folklore,
and indeed follows immediately from the $\ell$-group equality $x = (x\vee
e)(x\wedge e)$, we state it here.  

\begin{prop}\label{aut-commut}
Let $\mathbf{L}_1$ and $\mathbf{L}_2$ be $\ell$-groups, and
$\lambda_1$ and $\lambda_2$ automorphisms of $\mathbf{L}_1$   
and $\mathbf{L}_2$. If $f\colon\mathbf{L}_1\rightarrow \mathbf{L}_2$
is a homomorphism such that $(f\circ \lambda_1)(x) = (\lambda_2\circ f)(x)$ 
for all $x\in L_1^+\cup L_1^-$, then
$(f\circ \lambda_1)(x) = (\lambda_2\circ f)(x)$
holds for all $x\in L_1$.
\end{prop}

The variety $\var{\Psi MV}$ of \emph{pseudo MV-algebras} is 
a subvariety of $\mathsf{FL_w}$ defined by~\eqref{eq:luk}.
A fundamental result in the theory of pseudo MV-algebras, due to
Dvure\v censkij~\cite{Dvu02}, is  
that every pseudo MV-algebra is isomorphic to an algebra
$$
\Gamma(\mathbf L,u^{-1}) = ([u^{-1},e];\wedge,
\vee,\odot,\lres,\rres,u^{-1},e)
$$
where 
$\mathbf{L}$ is an $\ell$-group, $u\in L$ is a strong unit, 
$[u^{-1},e]=\{a\in L\mid u^{-1}\leq a\leq e\}$, the operations $\vee$ and $\wedge$
are inherited from $\mathbf{L}$ and the other operations are defined by
$$
x\odot y\dfeq x y\vee u^{-1},\quad x\lres y \dfeq x^{-1} y\wedge
e,\quad y\rres x\dfeq y x^{-1}\wedge e.
$$ 
Indeed, $\Gamma$ is precisely the Chang-Mundici functor applied in a non-Abelian
setting and Dvure\v{c}enskij~\cite{Dvu02} shows that it has an
appropriate adjoint.  

Pseudo MV-algebras were originally defined and studied by
Georgescu, Ior\-gu\-les\-cu~\cite{GI99} and~\cite{GI01},
as algebras $(A;\oplus,\fn,\sn,0,1)$ satisfying the identities:
\begin{itemize}
\item[(A1)] $x\oplus (y\oplus z) = (x\oplus y)\oplus z$,
\item[(A2)] $x\oplus 0=x$,
\item[(A3)] $x\oplus 1=1$,
\item[(A4)] $(x\fn\oplus y\fn)\sn=(x\sn\oplus y\sn)\fn$,
\item[(A5)] $(x\oplus y\sn)\fn \oplus x\fn = y\oplus(x\fn\oplus y)\sn$,
\item[(A6)] $x\oplus (y\fn\oplus x)\sn = y\oplus (x\fn\oplus y)\sn$ 
\item[(A7)] $x\fn\sn =x$,
\item[(A8)] $0\fn=1$.
\end{itemize}
The identities $0\oplus x = x$ and $1\oplus x = 1$ follow, as well as
$1\fn=0=1\sn$, and   $x\sn\fn =x$.
Defining $x\odot y\dfeq (x\fn\oplus y\fn)\sn$ one can show that
the identities below also hold.
\begin{align*}
x \odot 0 &= 0 = 0\odot x,\\
x \odot 1 &= x = 1\odot x,\\  
x\oplus(y\odot x\sn) &= y\oplus(x\odot y\sn)
 =(x\fn\odot y)\oplus x=(y\fn\odot x)\oplus y,\\ 
(x\fn\oplus y)\odot x &= y\odot (x\oplus y\sn).
\end{align*}
                        
Although we work in the setting of residuated lattices, some of the literature we cite
and some calculations we carry out, especially in the proof of Theorem~\ref{KitesRep},
use the original definition, so we recall the term equivalence between the two. 
In any pseudo MV-algebra defined by (A1)--(A8), the lattice
operations, multiplication and residuals are defined by 
\begin{align*}
x\vee y &\dfeq  x\oplus(y\odot x\sn), & x\wedge y &\dfeq
  (x\fn\oplus y)\odot x,\\
x\cdot y &\dfeq (x\fn\oplus y\fn)\sn, & & \\
x\lres y &\dfeq y\oplus x\sn, &  y\rres  x &\dfeq  x\fn\oplus y.
\end{align*}
Lattice order is then defined by any of the following mutually
equivalent conditions:  
$$
x\fn\oplus y=1,\quad  y\oplus x\sn=1,\quad x\odot y\sn=0, \quad y\fn\odot x=0.
$$
Conversely, we have:
$$
x\odot y \dfeq x\cdot y,\qquad x\oplus y\dfeq (x\fn\cdot y\fn)\sn,
$$ 
so the two definitions are term equivalent.

\section{Perfect residuated lattices}\label{sec:perfect-rls}

Perfect residuated lattices are a natural generalisation of \emph{perfect
  MV-algebras} defined in Belluce, Di Nola, Lettieri~\cite{BDL93}.
The authors define an MV-algebra $\mathbf{A}$ to be perfect if
for any element $a\in A$ exactly one of $a$ and $a\fn$ is of finite order,
where the order of an element $u$ is the least positive integer $m$
such that $u^m = 0$ or $\infty$ if no such $m$ exists (to be precise,
\cite{BDL93} defines a dual notion, but this form fits better in our notation).
The key structural property of
a perfect MV-algebra $\mathbf{A}$ is that there exists a surjective
homomorphism $h\colon \mathbf{A}\rightarrow \mathbf{2}$ such
that $h^{-1}(0)\leq h^{-1}(1)$, that is, if $h(x) = 0$ and $h(y) = 1$, then
$x\leq y$. Conversely, if this property holds in an MV-algebra $\mathbf{A}$,
then all members of $h^{-1}(0)$ are of finite order and
no member of $h^{-1}(1)$ is. However, the notion of finite order does not lend
itself easily to generalisations, whereas the structural property we just mentioned,
does. In fact, it is already general enough.

\begin{defin}\label{perfect}
An FL$_w$-algebra $\mathbf{A}$ will be called\/ \emph{perfect} if there
is a homomorphism $h_\mathbf A\colon \mathbf{A}\rightarrow\bm{2}$ such
that for any $x\in h_\mathbf A^{-1}(0)$ and any $y\in h_\mathbf A^{-1}(1)$ the
inequality $x\leq y$ holds.
\end{defin}
To spare notation, we put $F_\mathbf A \dfeq  h_\mathbf A^{-1}(1)$ and
$J_\mathbf A \dfeq  h_\mathbf A^{-1}(0)$, whenever $h$ is clear from context. 
Clearly, $F_\mathbf A$ is a maximal normal filter and
$J_\mathbf A = A\setminus F_{\mathbf{A}}$ is a lattice ideal.

For MV-algebras or symmetric
pseudo MV-algebras, Definition~\ref{perfect} is equivalent to
the original definition. However, it is immediately clear that our definition
applies to  any double pointed algebra in any signature $\tau$ such that Boolean
algebras are term equivalent to $\tau$-algebras. Take for example the extended
natural numbers $\overline{\mathbb{N}} = \mathbb{N}\cup\{\infty\}$, in the
signature of ordered semirings;  then $\overline{\mathbb{N}}$ is perfect.

\begin{lemma}\label{hom-unique}
Let $\mathbf{A}$ be a perfect FL$_w$-algebra. Then the homomorphism
$h_\mathbf A\colon \mathbf A\rightarrow\bm 2$ is unique. Hence
$\mathbf{A}$ has a unique maximal normal filter.
\end{lemma}  

\begin{proof}
Suppose $h,h'\colon \mathbf A\rightarrow\bm 2$ are homomorphisms, and
consider $F_\mathbf{A}$ and  $F'_\mathbf{A}$. To get a contradiction,
suppose $x\in F_\mathbf{A}\setminus F'_\mathbf{A}$. Then
$x\leq y$ holds for all $y\in F'_\mathbf{A}$, and therefore
$F'_\mathbf{A}\subseteq F_\mathbf{A}$. By maximality,
$F'_\mathbf{A} = F_\mathbf{A}$ contradicting the choice of $x$. 
\end{proof}

For an MV-algebra, having a unique maximal normal filter is precisely the
property of being 
\emph{local}, also defined in Belluce, Di Nola, Lettieri~\cite{BDL93}, so both
concepts (perfectness and locality) generalise smoothly to FL$_w$-algebras. 
The next lemma, whose easy proof we leave to the reader, restates in this
setting the observations made in Di Nola, Lettieri~\cite{DL94} 
for perfect MV-algebras. The nontriviality assumption for homomorphic images is
only necessary because the trivial algebra is not perfect, a fact overlooked
in Proposition 3.6 of~\cite{DL94}.

\begin{lemma}\label{HSPu-closure}
The class of perfect FL$_w$-algebras is closed under
nontrivial homomorphic images, subalgebras and ultraproducts.
\end{lemma}

The class of perfect FL$_w$-algebras is not closed under direct products,
but it is closed under certain subdirect products which we will now define.
For any family $\{\mathbf{A}_i\}_{i\in I}$ of perfect residuated lattices
we define the algebra  
$$
\pprod_{i\in I}\mathbf{A}_i \leq \prod_{i\in I}\mathbf A_i 
$$
whose elements are those $x\in\prod_{i\in I}A_i$
for which there is $k\in\{0,1\}$ 
such that $f_{\mathbf A_i}(x(i))=k$ for every $i\in I$.  Then
$\pprod_{i\in I}\mathbf A_i$ is also perfect and the corresponding
homomorphism
$$
f_{\pprod_{i\in I}\mathbf A_i}\colon\pprod_{i\in I}\mathbf A_i\rightarrow \bm 2
$$
satisfies 
$$
f_{\pprod_{i\in I}\mathbf A_i} (x)=f_{\mathbf{A}_i}(x(i))
$$ for any $i\in I$. Analogously we define 
$$
\pprod_{I}\mathbf A \leq \mathbf{A}^I
$$ 
where $\mathbf A$ is a perfect residuated lattice and $I$ is a set.
We call $\pprod$ a \emph{perfect product} and use $X$ for the associated class
operator. We will make some use of it in Sections~\ref{sec:kites-and-ppmv}
and~\ref{sec:approximation}.

For any class $\mathcal{K}$ of FL$_w$-algebras,
we denote the subclass of all its perfect members
by $\mathcal{K}_{\pf}$.  
We say that a variety $\mathcal{V}$ of FL$_w$-algebras is
\emph{perfectly generated} if it is generated by its perfect members, that is, if
$\mathcal{V} = V(\mathcal{V}_{\pf})$. 

\begin{theorem}\label{PerfGenEqBase}
A subvariety $\mathcal{V}$ of $\var{FL_w}$ is  perfectly generated if and
only if $\mathcal{V}$ is nontrivial and satisfies the following identities: 
\begin{align}
\cpol{\alpha}(x\rres x\fn)\vee \cpol{\beta}(x\fn\rres x) &= 1,\\
  \cpol{\alpha}((x\vee x\fn)\cdot (y\vee y\fn))\fn &\leq 
     \cpol{\alpha}((x\vee x\fn)\cdot (y\vee y\fn)),\\ 
x\wedge x\fn &\leq  y\vee y\fn
\end{align}
for every $\mathbf{A}\in\mathcal{V}$ and all $\cpol{\alpha},\cpol{\beta}\in
\mathrm{cPol}(\mathbf{A})$. 
\end{theorem}

\begin{proof}
Assume $\mathcal{V}$ is perfectly generated. Let
$\mathbf{A}\in \mathcal{V}$ be perfect, and let $f_\mathbf
A\colon \mathbf A\rightarrow\bm 2$ be the homomorphism witnessing it.
\begin{enumerate}
\item Since $x\in F_{\mathbf A}$ or $x\fn\in F_{\mathbf A}$,
we have $\mathbf A\models (x\leq  x\fn$ or $x\fn\leq x)$.
All iterated conjugates of $1$ are equal to $1$, so we get
$\cpol{\alpha} (x\rres x\fn)=1$ or $\cpol{\beta} (x\fn\rres x)=1$, and therefore
$\cpol{\alpha}(x\rres x\fn)\vee \cpol{\beta}(x\fn\rres x)=1$
holds.

\item As $x\vee x\fn,y\vee y\fn\in F_{\mathbf A}$, we have
  $\cpol{\alpha}((x\vee x\fn)\cdot (y\vee y\fn))\in F_{\mathbf A}(1)$ which gives
  $\cpol{\alpha}((x\vee x\fn)\cdot (y\vee y\fn))\fn \leq\cpol{\alpha}((x\vee
  x\fn)\cdot (y\vee  y\fn))$ holds.

\item As $x\wedge x\fn \in J_{\mathbf A}$ and $y\vee y\fn\in
  F_{\mathbf A}$ for all $x,y\in A$ we get that $x\wedge x\fn \leq y\vee
  y\fn$ holds. 
\end{enumerate}

For the converse, assume $\mathcal{V}$ is nontrivial and satisfies (1)--(3).
We will show that every $\mathbf A\in \mathcal{V}$ is a subdirect product
of perfect algebras. Take $x\in A\setminus\{1\}$ and using Zorn Lemma 
find a maximal normal filter $F\subseteq A$ such that
$x\not\in F$.

First we show that $y\rres y\fn\in F$ or $y\fn\rres  y\in F$ holds for any
$y\in A$. Assume the contrary. Then $y\rres y\fn\notin F$ and
$y\fn\rres  y\not\in F$ for some $y\in F$, and
so $x\in F_{y\rres y\fn}$ and $x\in F_{y\fn\rres y}$. By Lemma~\ref{FILTER}
there exist $f_1,f_2\in F$ and  
$\cpol{\alpha_1},\dots,\cpol{\alpha_n},\cpol{\beta}_1,\dots,\cpol{\beta}_m\in
\mathrm{cPol}(\mathbf{A})$ such that 
$$
x\geq f_1\cdot \cpol{\alpha_1}(y\rres y\fn)\cdots\cpol{\alpha_n}(y\rres y\fn),
f_2\cdot \cpol{\beta}_1(y\fn\rres  y)\cdots\cpol{\beta}_m(y\fn\rres  y).
$$
Hence 
\begin{align*}
x &\geq  (f_1\cdot f_2\cdot \cpol{\alpha_1}(y\rres
         y\fn)\cdots\cpol{\alpha_n}(y\rres y\fn))\vee (f_1\cdot f_2\cdot
         \cpol{\beta}_1(y\fn\rres  y)\cdots\cpol{\beta}_m(y\fn\rres  y))\\ 
&\geq  f_1\cdot f_2\cdot (\cpol{\alpha_1}(y\rres
       y\fn)\cdots\cpol{\alpha_n}(y\rres y\fn)\vee  \cpol{\beta}_1(y\fn\rres
       y)\cdots\cpol{\beta}_m(y\fn\rres  y)) 
\end{align*}
The fact that
$x\vee y_1=1=x\vee y_2$ implies $x\vee (y_1\cdot y_2)=1$, and the
identity (1) yield  
$$
\cpol{\alpha_1}(y\rres y\fn)\cdots\cpol{\alpha_n}(y\rres y\fn)\vee
\cpol{\beta}_1(y\fn\rres  y)\cdots\cpol{\beta}_m(y\fn\rres  y) =1
$$
and hence $x\geq f_1\cdot f_2.$ Thus $x\in F$ which contradicts the choice of
$F$. Consequently $y\rres y\fn\in F$ or $y\fn\rres  y\in F$ holds and we obtain
$$
\mathbf A/F\models y\leq y\fn\text{ or } y\fn\leq y.\eqno{(*)}
$$
Let $K\dfeq \{a\vee a\fn\colon a\in A/F\}$. Take $u,w\in K$ so that
$u=a\vee a\fn$ and $w=b\vee b\fn$ for some $a,b\in A/F$.
Let $\cpol{\alpha}\in \mathrm{cPol}(\mathbf{A})$. Using (2) we get 
\begin{align*}
  \cpol{\alpha} (u\cdot w)\vee\cpol{\alpha} (u\cdot w)\fn
  &= \cpol{\alpha} ((a\vee a\fn)\cdot (b\vee b\fn))\vee\cpol{\alpha}
      ((a\vee a\fn)\cdot (b\vee b\fn))\fn\\ 
&= \cpol{\alpha} ((a\vee a\fn)\cdot (b\vee b\fn))\\
&= \cpol{\alpha} (u\cdot w).
\end{align*}
Hence $\cpol{\alpha} (u\cdot w)\in K$, so $K$ is a normal filter.
It is clear that $0\not\in K$, so $K$ is a proper filter, and
($*$) shows that $a\in K$ or $a\fn\in K$ for any $a\in A/F$.
Therefore $(\mathbf A/F)/K\cong\bm 2$. 

It remains to verify that if $u\in K$ and $w\notin K$, then $w\leq u$.
By ($*$) we have either $w\leq w\fn$ or  $w\fn\leq w$.
If $w\fn\leq w$, then $w = w\vee w\fn$ and so $w\in K$ contradicting its
choice. Hence $w\leq w\fn$.  On the other hand, $u\in K$ so
$u = a\vee a\fn$ for some $a\in A/F$ and so
$u\fn = (a\vee a\fn)\fn = a\fn\wedge a\fn\fn \leq a\vee a\fn = u$.
Therefore, using (3) we obtain $w = w\wedge w\fn\leq u\vee u\fn = u$
proving that $\mathbf A/F$ is perfect.

We have shown that for any $x \in A\setminus\{1\}$  there
is a normal filter $F\subseteq A$ such that $x\not\in F$ and $\mathbf A/F$
is a perfect FL$_w$-algebra. By the correspondence between normal filters and
congruences it follows that $\mathbf A$ is a subdirect product of perfect
algebras. 
\end{proof}

\begin{cor}\label{down-closed}
Any nontrivial subvariety of a perfectly generated variety is also perfectly generated.
\end {cor}

The variety $\var{BA}$ of Boolean algebras is the unique atom of the lattice
$\Lambda(\var{FL_w})$ of all subvarieties of $\var{FL_w}$. 
Since the trivial variety is not perfectly generated,
$\var{BA}$ is the smallest perfectly
generated variety, hence the poset $\Lambda^+(\var{FL_w})$ of nontrivial
subvarieties of $\var{FL_w}$ is a sublattice of $\Lambda(\var{FL_w})$. 
We will now show that perfectly generated varieties form a lattice ideal in 
$\Lambda^+(\var{FL_w})$.

\begin{lemma}\label{si-perfect}
Let $\mathcal{V}$ be a variety of FL$_w$-algebras. Then
$\mathcal{V}$ is perfectly generated if and only if the class
$\mathcal{V}_{si}$ of all subdirectly
irreducible members of $\mathcal{V}$ consists of perfect algebras.  
\end{lemma}

\begin{proof}
The right-to-left direction is clear. For the left-to-right direction, if 
$\mathcal{V}$ is perfectly generated and $\mathbf{A}\in\mathcal{V}$ is
subdirectly irreducible, then by J\'onsson's Lemma,
$\mathbf{A}\in HSP_U(\mathcal{V}_{\pf})$. Then $\mathbf{A}$ is perfect,
by Lemma~\ref{HSPu-closure}.
\end{proof}

\begin{theorem}\label{PerfGenIdeal}
Perfectly generated varieties form an ideal in $\Lambda^+(\mathsf{FL_w})$.  
\end{theorem}  

\begin{proof}
By combining Corollary~\ref{down-closed}, Lemma~\ref{si-perfect}, and the fact
that in congruence distributive varieties
$(\mathcal{V}_1\vee\mathcal{V}_2)_{si} =
(\mathcal{V}_1)_{si}\cup(\mathcal{V}_2)_{si}$.
\end{proof}

\section{Perfect pseudo MV-algebras}\label{sec:perfect-psmv}

Di Nola, Lettieri~\cite{DL94} showed that perfect MV-algebras are
categorically equivalent to Abelian $\ell$-groups. Di Nola, Dvure\v{c}enskij,
Tsinakis~\cite{DDT08} generalised this result, showing that
symmetric pseudo MV-algebras are categorically equivalent to $\ell$-groups.
In this section we show that
perfect pseudo MV-algebras are categorically equivalent to $\ell$-groups
with a distinguished automorphism, obtaining the results mentioned above as
corollaries. 

The algebras we will work with in this section, and indeed in the rest of the
article, are algebras $\mathcal K (\mathbf L,\lambda)$, constructed out of
an $\ell$-group $\mathbf{L}$ and an automorphism
$\lambda\colon\mathbf{L}\rightarrow\mathbf{L}$. Before we formally define them,
here are the quick-start assembling instructions. 
Take any $\ell$-group $\mathbf{L}$ and an automorphism $\lambda$. Throw away
everything outside $L^-\cup L^+$, take $L^-$ and $L^+$ apart, put $L^-$ on top of $L^+$,
define the operations using $\lambda$ to give $L^+$ a twist. See
Figure~\ref{pic:kite}.

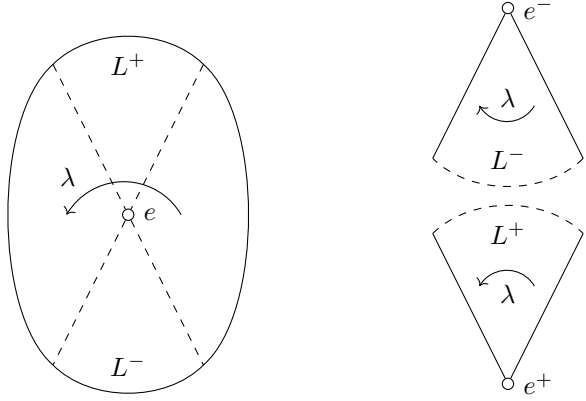
\begin{figure}
\begin{tikzpicture}
\node[odot, label=right:$e$] (v0) at (0,0) {};
\node[coordinate] (v1) at (-1,2) {};
\node[coordinate] (v2) at (1,2) {};
\node[coordinate] (v3) at (-1,-2) {};
\node[coordinate] (v4) at (1,-2) {};
\node (L-) at (0,-2) {$L^-$};
\node (L+) at (0,2) {$L^+$};
\node (l) at (-0.8,0.5) {$\lambda$};
\draw[dashed] (v1)--(v0)--(v4) (v2)--(v0)--(v3);
\draw (v1) .. controls (-0.5,2.5) and (0.5,2.5) .. (v2)
.. controls (1.8,1.3) and (1.8,-1.3) .. (v4)
.. controls (0.5,-2.5) and (-0.5,-2.5) .. (v3)
.. controls (-1.8,-1.3) and (-1.8,1.3) .. cycle;
\draw[->] (20pt,0) arc[start angle=30, end angle=150, radius=25pt]; 
\end{tikzpicture}  
\hspace{2cm}
\begin{tikzpicture}
\node[odot, label=right:$e^+$] (vb) at (0,-2.5) {};
\node[coordinate] (v1) at (-1,-0.5) {};
\node[coordinate] (v2) at (1,-0.5) {};
\node[odot, label=right:$e^-$] (vt) at (0,2.5) {};
\node[coordinate] (v3) at (-1,0.5) {};
\node[coordinate] (v4) at (1,0.5) {};
\node (L-) at (0,0.5) {$L^-$};
\node (L+) at (0,-0.5) {$L^+$};
\node (l1) at (0,1.3) {$\lambda$};
\node (l2) at (0,-1.3) {$\lambda$};
\draw (vb)--(v1) (vb)--(v2);
\draw (vt)--(v3) (vt)--(v4);
\draw[dashed] (v1) .. controls (-0.5,0) and (0.5,0) .. (v2)
(v4).. controls (0.5,0) and (-0.5,0) .. (v3);
\draw[->] (10pt,-1.2) arc[start angle=30, end angle=150, radius=12pt];
\draw[->] (10pt,1.2) arc[start angle=-30, end angle=-150, radius=12pt]; 
\end{tikzpicture}  
\caption{Constructing $\mathcal{K}(\mathbf{L},\lambda)$.}\label{pic:kite}
\end{figure}

\begin{defin}\label{gen-kite}
Let $\mathbf L$ be an $\ell$-group
and $\lambda\colon \mathbf L\rightarrow\mathbf L$ be an automorphism. We
define the algebra
$$
\mathcal K (\mathbf L,\lambda) \dfeq (L^-\uplus
L^+;\wedge,\vee,\odot,\lres,\rres ,0,1)
$$
where $L^-\uplus L^+$ is a disjoint union, 
$0\dfeq e \in L^+$, $1\dfeq e\in L^-$,
and the other operations are given by

\begin{align*}
x\wedge y &\dfeq \begin{cases}
x\wedge y \in L^- & \text{ if } x,y \in L^-,\\
x\in L^+  &  \text{ if } x\in L^+, y \in L^-\\
y\in L^+ & \text{ if }  x\in L^-, y \in L^+,\\
x\wedge y \in L^+ & \text{ if } x,y \in L^+,
\end{cases}\\
x\vee y &\dfeq \begin{cases}
x\vee y \in L^- &\text{ if }  x,y \in L^-,\\
y \in L^- & \text{ if }  x\in L^+, y \in L^-\\
x\in L^-& \text{ if }  x\in L^-, y \in L^+,\\
x\vee y \in L^+ & \text{ if }  x,y \in L^+,
\end{cases}\\
x\odot y &\dfeq \begin{cases}
x\cdot y \in L^- &\text{ if }  x,y \in L^-,\\
\lambda(x)\cdot y\vee e\in L^+ & \text{ if }  x\in L^-, y \in L^+\\
x\cdot y\vee e \in L^+& \text{ if }  x\in L^+, y \in L^-,\\
 e \in L^+ & \text{ if }  x,y \in L^+,
\end{cases}\\
x\lres y & \dfeq \begin{cases}
x^{-1}\cdot y \wedge e\in L^-&\text{ if }  x,y \in L^-,\\
e \in L^- & \text{ if }  x\in L^+, y \in L^-\\
\lambda (x)^{-1}\cdot y\vee e \in L^+& \text{ if }  x\in L^-, y \in L^+,\\
 x^{-1}\cdot y\wedge e\in L^- & \text{ if }  x,y \in L^+,
\end{cases}\\
y\rres x &\dfeq \begin{cases}
y\cdot x^{-1} \wedge e\in L^-&\text{ if }  x,y \in L^-,\\
e\in L^- & \text{ if }  x\in L^+, y \in L^-\\
 y\cdot x^{-1}\vee e\in L^+& \text{ if }  x\in L^-, y \in L^+,\\
 \lambda^{-1}(y\cdot x^{-1})\wedge e\in L^- & \text{ if }  x,y \in L^+,
\end{cases}
\end{align*}
\end{defin}

\begin{remark}\label{negations}
The negations $x\fn\dfeq  0\rres x$ and  $x\sn\dfeq  x\lres 0$
in $\mathcal K (\mathbf L,\lambda)$ are given by
\begin{align*}
x\fn &= \begin{cases}
 x^{-1} \in L^+&\text{ if }  x \in L^-,\\
\lambda^{-1} (x)^{-1}\in L^- & \text{ if } x\in L^+.
\end{cases}\\
x\sn &= \begin{cases}
\lambda (x)^{-1}\in L^+& \text{ if } x \in L^-,\\
 x^{-1}\in L^- & \text{ if } x\in L^+.
\end{cases}
\end{align*}
\end{remark}

% If $\mathbf{L}$ is non-Abelian, then $\mathcal{K}(\mathbf{L}, id_L)$ is
% non-commutative, of course, but the two negations coincide. Below we give two
% examples of non-commutative $\mathcal{K}(\mathbf{L}, \lambda)$ obtained from a
% commutative  $\mathbf{L}$. 

% \begin{example}
% Let $\mathbf{L} = \mathbb{Z}\times \mathbb{Z}$, and let $\lambda$ be given
% by $\lambda(i,j) = (j,i)$. Then, $\mathcal{K}(\mathbf{L}, \lambda)$  
% \end{example}  

\begin{theorem}\label{GenKite}
Let $\mathbf L$ be an $\ell$-group
and $\lambda\colon \mathbf L\rightarrow\mathbf L$ an automorphism. Then
$\mathcal K (\mathbf L,\lambda)$
is a perfect pseudo MV-algebra with $J_{\mathcal K(\mathbf L,\lambda)}=L^+$ and
$F_{\mathcal K (\mathbf L,\lambda)}=L^-$. Moreover:
\begin{enumerate}
\item $\mathcal K(\mathbf L,\lambda)$
is a symmetric perfect pseudo MV-algebra if and only if\/ $\lambda = id_L$.
\item $\mathcal K(\mathbf L,\lambda)$
is a perfect MV-algebra if and only if\/ $\mathbf{L}$ is Abelian and
$\lambda = id_L$.
\end{enumerate}
\end{theorem}

\begin{proof}
It is not difficult to verify that
$\mathcal K(\mathbf L,\lambda) =
\Gamma(\mathbf{L}\ltimes_{\mu}\mathbb Z,(e,-1))$ where  
$\mathbf{L}\ltimes_{\mu}\mathbb Z$ is the antilexicographically ordered
semidirect product of $\mathbf{L}$ and $\mathbb{Z}$ with respect to
the automorphism $\mu\dfeq\lambda^{-1}$. The group operations are explicitly given by
\begin{align*}
(x,m)\cdot(y,n)&\dfeq  (\lambda^{-n}(x)\cdot y,m+n),\\
(x,m)^{-1} &\dfeq  (\lambda^{m}(x^{-1}),-m),
\end{align*}
where we adopt the convention that $\lambda^0 = id$ (with this
convention $\lambda^{m}(x^{-1}) = \lambda^{-m}(x)$ whenever $m\neq 0$). 

The backward directions of (1) and (2) follow immediately
from the definitions. For the forward direction of (1), suppose
$\mathcal K(\mathbf L,\lambda)$ satisfies $x\fn = x\sn$. Then,
$\lambda(x) = x$ holds on $L^+\cup L^-$; hence
$\lambda(x) = x$ on $L$. For the forward direction of (2), suppose
$\mathcal K(\mathbf L,\lambda)$ is commutative.
Then by definition of $\odot$ the product in $\mathbf{L}$
commutes on $L^-$. This implies that it commutes on the
whole universe, so $\mathbf{L}$ is Abelian. Commutativity of
$\mathcal K(\mathbf L,\lambda)$ also implies
$x\ld y = y\rd x$, so taking any $x\in L^-$ and $y = e\in L^+$ we get
$$
x\ld e = \lambda(x)^{-1}\cdot e\vee e = \lambda(x)^{-1}\cdot e =
\lambda(x)^{-1} = \lambda(x^{-1})
$$
where the second equality holds because $\lambda$ fixes the negative cone setwise
and so $\lambda(x)^{-1}\in L^+$. Similarly, we obtain 
$$
e\rd x = e\cdot x^{-1}\vee e = e\cdot x^{-1} = x^{-1} 
$$
and thus $\lambda$ is the identity on the positive cone. Hence
$\lambda$ is the identity on the whole universe. 
\end{proof}

In any perfect pseudo MV-algebra $\mathbf A$ the normal filter $F_\mathbf A$
is a subuniverse of its residuated lattice reduct $\mathbf A\!^{rl}$; furthermore,
the subalgebra $\mathbf{F}_\mathbf{A}$ of $\mathbf A\!^{rl}$ 
is a cancellative IGMV-algebra. Since pseudo MV-algebras satisfy the identities
\begin{align*}
(x\wedge y)\sn\sn &=x\sn\sn \wedge y\sn\sn\\ 
(x\vee y)\sn\sn &=x\sn\sn \vee y\sn\sn\\ 
(x\cdot y)\sn\sn &=x\sn\sn \cdot y\sn\sn\\ 
x\fn\sn\sn &=x\sn\sn\fn
\end{align*}
the map $-\sn\sn$ is an automorphism of $\mathbf F_\mathbf A$.
Applying the functor $\ell$ we lift $-\sn\sn$ to an automorphism  
$$
\ell^\approx\colon \ell (\mathbf F_\mathbf A)\rightarrow
\ell (\mathbf F_\mathbf A)
$$
defined, obviously, as $\ell^\approx(-) \dfeq \ell(-\sn\sn)$.
Functoriality of $\ell$ also ensures that every algebra
$\mathbf{A}\in\var{\cigmv}$ comes with an ambient $\ell$-group
$\ell(\mathbf{A})$. In particular, for every perfect pseudo MV-algebra $\mathbf{A}$
we have an associated $\ell$-group $\ell(\mathbf{F}_{\mathbf{A}})$, so for an
element $x\in F_{\mathbf{A}}$, the notation $x^{-1}$ makes sense,
since the set $F_\mathbf A$ is a subset of both $A$ and $|\ell(\mathbf F_\mathbf A)|$.
Similarly, for $x\in J_{\mathbf{A}}$, we may write $(x\fn)^{-1}$ or 
$(x\sn)^{-1}$ with the inner operation taken in $\mathbf{A}$ and the outer
inverse in $\ell(\mathbf{F}_{\mathbf{A}})$.
We will use this notation frequently from now on.

\begin{theorem}\label{KitesRep}
Let $\mathbf A$ be a perfect pseudo MV-algebra. Then
$\mathbf A \cong\mathcal K(\ell (\mathbf F_\mathbf A),\ell^\approx)$.
\end{theorem}

\begin{proof}
Define a map  
$\Omega\colon A\rightarrow |\mathcal K(\ell (\mathbf F_\mathbf
A),\ell^\approx)|=\ell (\mathbf F_\mathbf A)^+\uplus\ell (\mathbf F_\mathbf A)^-$
by putting
$$
\Omega (x) =\begin{cases}
 x & \text{ if }  x \in F_\mathbf A\\
 (x\sn)^{-1} & \text{ if }  x \in J_\mathbf A.
\end{cases}
$$
We will show that $\Omega$ is an isomorphism. 

First we prove that $\Omega$ preserves negations.
If $x\in F_\mathbf A$, then $x\fn,x\sn\in J_\mathbf A$ and we have
\begin{align*}
\Omega (x\fn) &= (x\fn\sn)^{-1} = x^{-1} = \Omega (x)^{-1} = \Omega (x)\fn, \\
\Omega (x\sn) &= (x\sn\sn)^{-1} = \ell^\approx(x)^{-1} =
                \ell^\approx(\Omega(x))^{-1} = \Omega(x)\sn.
\end{align*}
If $x\in J_\mathbf A$, then $x\fn,x\sn\in F_\mathbf A$ and
\begin{align*}
\Omega
(x\fn) &= x\fn = x\sn\fn\fn = ((\ell^{\approx})^{-1}((x\sn)^{-1}))^{-1} =
((\ell^{\approx})^{-1}(\Omega (x)))^{-1} = \Omega (x)\fn,\\
\Omega(x\sn) &= x\sn = (x\sn)^{-1-1} = \Omega (x)^{-1} = \Omega (x)\sn.
\end{align*}
Next we prove that $\Omega$ preserves products. Recall that products in
$\mathbf F_\mathbf A$ and $\ell(\mathbf F_\mathbf A)$ are denoted
by $\cdot$ (and they coincide on $F_A$), whereas products in  $\mathcal K(\ell (\mathbf
F_\mathbf A),\ell^\approx)$ are denoted by $\odot$. 
Taking any $x,y\in A$, we have four cases to deal with:
\begin{itemize}
\item[(i)] If $x,y\in F_\mathbf A$ then $x\cdot y\in F_\mathbf A$ and hence
$$
\Omega (x\cdot y)=x\cdot y =\Omega(x)\odot\Omega (y)
$$
holds.

\item[(ii)] If $x\in F_\mathbf A$ and $y\in J_\mathbf A$ then
$x\cdot y\in J_\mathbf A$.
To begin with assume that $y\fn\leq x$, and so $y\sn\leq x\sn\sn$. Then 
\begin{align*}
 (x\cdot y)\sn\cdot x\sn\sn
&= (x\sn\oplus y\sn)\cdot x\sn\sn\\
&= (x\sn\sn\fn\oplus y\sn)\cdot x\sn\sn=x\sn\sn\wedge y\sn=y\sn
\end{align*}
and therefore, taking inverses in $\ell(\mathbf{F}_\mathbf{A})$ we get
$$
((x\cdot y)\sn)^{-1} = x\sn\sn \cdot (y\sn)^{-1}.
$$
Since $x\sn\sn,y\sn,(x\cdot y)\sn\in F_\mathbf A\subseteq \ell (\mathbf
F_\mathbf A)$ we obtain 
\begin{align*}
\Omega (x\cdot y) &= ((x\cdot y)\sn)^{-1} =(x\sn\sn)\cdot (y\sn)^{-1}\\
 &= \ell^\approx\Omega (x)\cdot \Omega (y) =\Omega (x)\odot \Omega (y).
\end{align*} 
The map $\Omega$ trivially preserves lattice operations and products for elements
of $F_\mathbf A$, so because $x,y\fn\in F_\mathbf A$, and obviously
$y\fn\leq x\vee y\fn$, we get in full generality: 
\begin{align*}
\Omega (x\cdot y) &= \Omega((x\vee y\fn)\cdot y)=\Omega((x\vee y\fn))\odot \Omega(y)\\
 &= (\Omega(x)\vee \Omega(y)\fn)\odot \Omega(y)=\Omega(x)\odot\Omega (y).
\end{align*}

\item[(iii)] If $x\in J_\mathbf A$ and $y\in F_\mathbf A$ then $x\cdot y\in
  J_\mathbf A$.
Analogously to the proof of (ii), assume first that $x\sn\leq y$. Then
$$
y\cdot (x\cdot y)\sn = y \cdot (x\sn \oplus y\sn) = x\sn \wedge y= x\sn
$$
and therefore
$$
((x\cdot y)\sn)^{-1} = (x\sn)^{-1}\cdot y.
$$
Hence we obtain
\begin{align*}
\Omega (x\cdot y) &= ((x\cdot y)\sn)^{-1} = (x\sn)^{-1}\cdot y\\
&= \Omega (x)\odot\Omega (y).
\end{align*}
Trivially $x\sn\leq x\sn\vee y$ holds, hence we get in full generality
\begin{align*}
\Omega (x\cdot y) &= \Omega(x\cdot (x\sn\vee y))=\Omega(x)\cdot
                    (\Omega(x)\sn\vee\Omega (y))\\ 
 &= \Omega (x)\odot\Omega (y).
\end{align*}

\item[(iv)] If $x,y\in J_\mathbf A$ then $\Omega (x\cdot y)= \Omega (1)=1=\Omega
  (x)\odot\Omega (y)$. 
\end{itemize} 
Since $\Omega$ preserves the constants $0$ and $1$, the product and the
negations, it is a homomorphism. Bijectivity follows immediately from
the definition of $\Omega$. 
\end{proof}

\begin{lemma}\label{KiteHom}
Let $\mathbf L_1$ and $\mathbf L_2$ be $\ell$-groups and $\lambda_1\colon
\mathbf L_1\rightarrow\mathbf L_1$, $\lambda_2\colon \mathbf
L_2\rightarrow\mathbf L_2$ be automorphisms. Then any homomorphism $f\colon
\mathbf L_1\rightarrow\mathbf L_2$ such that $f\circ\lambda_1 =
\lambda_2\circ f$ induces the homomorphism 
$$
h_f\colon \mathcal{K}(\mathbf L_1,\lambda_1)\rightarrow
\mathcal{K}(\mathbf L_2,\lambda_2)
$$ 
where
$h_f = f|_{L_1^-\cup L_1^+}$. Moreover, every homomorphism
$h\colon \mathcal{K}(\mathbf L_1,\lambda_1)\rightarrow
\mathcal{K}(\mathbf L_2,\lambda_2)$ arises this way.
\end{lemma}

\begin{proof}
The main part follows by inspection of definitions of operations in
$\mathcal{K}(\mathbf L_i,\lambda_i)$ for $i=1,2$. We sketch two cases.
Let $x\in L^-$ and $y\in L^+$. Consider $x\wedge y$ in
$\mathcal{K}(\mathbf L_1,\lambda_1)$. We have 
$$
h_f(x\wedge y) = h_f(y) = f(y) =
  f(x) \wedge f(y) = h_f(x)\wedge h_f(y),
$$
where $f(x) \wedge f(y)$ is taken in $\mathcal{K}(\mathbf L_2,\lambda_2)$;
hence the apparently reversed lattice order. Next, consider $x\ld y$. We have
\begin{align*}
h_f(x\ld y) &= h_f(\lambda_1(x)^{-1}y\vee e)\\
&= f(\lambda_1(x)^{-1}y\vee e) \\ 
&= \lambda_2(f(x^{-1}))\cdot f(y)\vee e\\
&= h_f(x)\ld h_f(y).
\end{align*}

For the moreover part,
let $h\colon \mathcal{K}(\mathbf L_1,\lambda_1)\rightarrow
\mathcal{K}(\mathbf L_2,\lambda_2)$ be a homomorphism.
Then $h|_{L_1^-}\colon \mathbf{L}_1^-\rightarrow \mathbf{L}_2^-$ is a
homomorphism of $\var{\cigmv}$. Applying the functor $\ell$, we obtain
a homomorphism $\ell(h|_{L_1^-})\colon \ell(\mathbf{L}_1^-)\rightarrow
\ell(\mathbf{L}_2^-)$. Since $\ell(\mathbf{L}_i^-)\cong\mathbf{L}_i$ 
for $i=1,2$, there is a homomorphism
$g\colon \mathbf{L}_1\rightarrow\mathbf{L}_2$ such that
$g|_{L_1^-} = h|_{L_1^-}$. We need to show that $g|_{L_1^+} = h|_{L_1^+}$. 
Take any $y\in L_1^+$. Since negations are bijective in pseudo MV-algebras,
in $\mathcal{K}(\mathbf L_1,\lambda_1)$ we have
$y = x\fn$ for some $x$. By definition of operations in $\mathcal{K}(\mathbf
L_1,\lambda_1)$, we get further that $x\in L_1^-$ and $x^- = x^{-1}$. Hence,
$g(y) = g(x^{-1}) = g(x)^{-1} = h(x)^{-1} = h(x)\fn = h(x\fn) = h(y)$, as required.

It remains to show that $g$ commutes with the automorphisms $\lambda_1$, $\lambda_2$.
By Remark~\ref{negations}, in $\mathcal{K}(\mathbf{L}_1,\lambda_1)$ we have
$x\sn\sn = \lambda_1(x)$ for all $x\in L_1^-$, so taking any $x\in L_1^-$ we calculate
$(g\circ\lambda_1)(x) = g(x\sn\sn) = h(x\sn\sn) = h(x)\sn\sn =
g(x)\sn\sn = (\lambda_2\circ g)(x)$. By Proposition~\ref{aut-commut} we then have
$(g\circ\lambda_1)(x) = (\lambda_2\circ g)(x)$ for all $x\in L_1$, as required.
\end{proof}

The next nearly trivial fact will be important later, so we state it
explicitly.    

\begin{lemma}\label{KiteFilHom}
Let $f$, $\mathbf{L}_i$, $\lambda_i$, for $i=1,2$ be as in Lemma~\ref{KiteHom}.
For any $x\in J_{\mathcal{K}(\mathbf L_1,\lambda_1)}$, we have
$f(x) = f(x\sn)^{-1}$.
\end{lemma}  

\begin{proof}
By definition of $\mathcal{K}(\mathbf L_1,\lambda_1)$, we have
$J_{\mathcal{K}(\mathbf L_1,\lambda_1)} = L_1^+$.
Then, $f(x\sn)^{-1} = f(x^{-1})^{-1} = f(x^{-1-1}) = f(x)$. 
\end{proof}

Now we are ready for a proof of the categorical equivalence between
perfect pseudo MV-algebras and $\ell$-groups that we announced at the beginning
of the section. 

\begin{defin}\label{aut-l-groups}
We define $\mathsf{LGA}$ to be the category of $\ell$-groups with a
distinguished automorphism. The objects are algebras
$(\mathbf{L},\lambda)$ where $\mathbf{L}$ is an $\ell$-group and $\lambda$ is 
an automorphism of\/ $\mathbf{L}$. The morphisms are $\ell$-group homomorphisms
commuting with the distinguished automorphism.
\end{defin}

The category $\mathsf{LGA}$ is not a variety, but
it is obviously categorically equivalent to a variety of algebras
$(\mathbf{L},\lambda,\mu)$, where $\mathbf{L}$ is an $\ell$-group and
$\lambda$, $\mu$ are endomorphisms of $\mathbf{L}$ satisfying
$\lambda(\mu(x)) = x =  \mu(\lambda(x))$.

\begin{theorem}\label{cat-eqv-ppMV-LGA}
The categories $\var{\pfpmv}$ of perfect pseudo MV-algebras, and
$\mathsf{LGA}$ of $\ell$-groups with a distinguished automorphism, are
equivalent.    
\end{theorem}

\begin{proof}
Lemma~\ref{KiteHom} shows that $\mathcal{K}(-)$ is a faithful and full functor from
$\mathsf{LGA}$ to $\mathsf{\pfpmv}$.  
Theorem~\ref{KitesRep} shows that $\mathcal{K}(-)$ is essentially surjective.
\end{proof}  

By categorical equivalence between $\var{\cigmv}$ and $\var{LG}$, we
immediately obtain the next result.

\begin{cor}
The categories $\var{\pfpmv}$, and
$\var{\cigmv{}A}$ of cancellative IGMV-algebras with a distinguished automorphism, are
equivalent.
\end{cor}  

As we mentioned a few times already, 
Di Nola, Lettieri~\cite{DL94} showed that perfect MV-algebras are
categorically equivalent to Abelian $\ell$-groups, and 
Di Nola, Dvure\v{c}enskij, Tsinakis~\cite{DDT08} generalised the result
to an equivalence between symmetric perfect pseudo
MV-algebras and $\ell$-groups. Now we can obtain these results as corollaries.

\begin{cor}
The following pairs of categories are equivalent:
\begin{enumerate}
\item Symmetric perfect pseudo MV-algebras and $\ell$-groups. 
\item Perfect MV-algebras and Abelian $\ell$-groups.
\end{enumerate}
\end{cor}

\begin{proof}
By Theorems~\ref{GenKite} and~\ref{cat-eqv-ppMV-LGA},
and transitivity of categorical equivalence.
\end{proof}

\section{Kites and perfect pseudo MV-algebras}\label{sec:kites-and-ppmv} 
 
For any power $\mathbf{L}^B$ of an $\ell$-group $\mathbf{L}$,
a very natural automorphism $\lambda\colon \mathbf{L}^B\rightarrow
\mathbf{L}^B$ is induced by any bijection 
$\beta\colon B\rightarrow B$ by taking
$\lambda(x)(i)\dfeq x(\beta(i))$ for any $i\in B$.
By Theorem~\ref{GenKite}, $\mathcal{K}(\mathbf{L}^B,\lambda)$ is
a perfect pseudo MV-algebra, indeed, it is a particular case of a
\emph{kite} from Dvure\v{c}enskij, Kowalski~\cite{DK14}. 
Henceforth, kites will be our main focus, so
we now introduce some machinery to deal with them.

\begin{defin}\label{cycles}
A monounary algebra $\mathbf B=(B;\beta)$ where
$\beta$ is a bijection on $B$  will be called
a \emph{B-cycle}.
Homomorphisms of B-cycles are maps $f\colon \mathbf B\rightarrow
\mathbf C$ satisfying $f\circ \lambda^{\mathbf{B}}=\lambda^{\mathbf{C}}\circ f$.
Objects of the category $\var{BC}$
are B-cycles and arrows are homomorphisms.  
\end{defin} 

\begin{remark}
B-cycles are not a variety, but as we will often need $\beta^{-1}$,
we could have equivalently defined B-cycles as a variety of
bi-unary algebras $(B,\beta,\delta)$ satisfying
$\beta(\delta(x)) = x = \delta(\beta(x))$, and write $\beta^{-1}$ for $\delta$.
\end{remark}

Precisely this variety was used in Baldwin, Berman~\cite{BB75}
as the first example of a variety which has arbitrarily large finite subdirectly
irreducibles, but no infinite ones.
%We will view $\var{BC}$ as a variety in
%Lemma~\ref{BcVarLatt}. 

\begin{defin}\label{classical-kites}
Let $\mathbf{B} = (B;\beta)$ be a B-cycle and $\mathbf{L}$ and $\ell$-group. 
A \emph{kite} over $\mathbf{B}$ and $\mathbf{L}$ is the algebra
$$
\mathcal{K}_{\mathbf{B}}(\mathbf{L}) \dfeq
\mathcal{K}(\mathbf{L}^B,\lambda)
$$
where $\lambda\colon \mathbf{L}^B\rightarrow \mathbf{L}^B$ is the
automorphism given by $\lambda(x)(i) = x(\beta(i))$ for any $i\in B$.
\end{defin}  

Below we state the precise relationship between kites of
Definition~\ref{classical-kites} and kites in the sense of
Dvure\v{c}enskij, Kowalski~\cite{DK14},
leaving the verification to the interested reader.

\begin{prop}\label{kites-and-kites}
Let $\mathbf{L}$ be an $\ell$-group, and $\mathbf{B} = (B;\beta)$ a B-cycle.
Then $\mathcal{K}_{\mathbf{B}}(\mathbf{L})$
is isomorphic to $K^{\beta,id}_{B,B}(\mathbf{L})$ from~\cite{DK14}.
\end{prop}

\begin{example}\label{not-DvuKow}
Let $\mathbb{Q}$ be the additive $\ell$-group of the rationals. The map
$\lambda(x) = 2x$ is an automorphism of $\mathbb{Q}$, so
$\mathcal{K}(\mathbb{Q}, \lambda)$ is a perfect
pseudo MV-algebra. But $\mathcal{K}(\mathbb{Q}, \lambda)$
is not a kite, as Lemma~\ref{not-kite} will show.
\end{example}

\begin{lemma}\label{FilHom}
Let $\mathbf A$ and $\mathbf B$ be perfect pseudo MV-algebras and let
$f\colon\mathbf{F}_\mathbf A\rightarrow\mathbf F_\mathbf B$ be a homomorphism of
IGMV-algebras. If $f(x\sn\sn )=f(x)\sn\sn$ holds for all $x\in F_\mathbf A$,
then there exists a homomorphism $\overline{f}\colon \mathbf A\rightarrow\mathbf B$,
given by
$$
\overline{f}(x) = \begin{cases}
f(x) & \text{ if } x\in F_\mathbf{A}\\
f(x\sn)^{-1} & \text{ if } x\in J_\mathbf{A}.
\end{cases}  
$$
If $f$ is surjective/injective, so is $h$.
\end{lemma}

\begin{proof}
We have a unique homomorphism
$\ell(f)\colon \ell(\mathbf{F_A})\rightarrow \ell(\mathbf{F_B})$.
By the remarks preceding Theorem~\ref{KitesRep}, the map
$\ell^\approx(-) = \ell(-^{\sn\sn})$ is an automorphism
of both $\ell(\mathbf{F_A})$ and $\ell(\mathbf{F_B})$.
Since $f$ commutes with $(-^{\sn\sn})$, we have
$\ell(f)\circ\ell^\approx = \ell^\approx\circ \ell(f)$ and
the claim follows by Lemmas~\ref{KiteHom} and~\ref{KiteFilHom}.
The moreover part is immediate.
\end{proof}

\begin{lemma}\label{KiteHoms}
Let\/ $\mathbf B$ be a B-cycle and $f\colon \mathbf L_1\rightarrow\mathbf L_2$
be a homomorphism of $\ell$-groups. Then there exists a homomorphism  
$$
\mathcal K_\mathbf B(f)\colon \mathcal K_\mathbf B(\mathbf L_1)\rightarrow
\mathcal K_\mathbf B(\mathbf L_2)
$$
such that $((\mathcal K_\mathbf{B}(f))(x))(i)=f(x(i))$
for any $x\in F_{\mathcal K_\mathbf B(\mathbf{L}_1)}=(L_1^-)^B$ and any $i\in B$. 
\end{lemma}

\begin{proof}
It is easily seen that $f$ induces a homomorphism
$$
\tilde{f}\colon (\mathbf L_1^-)^B\rightarrow (\mathbf L_2^-)^B
$$
defined by $(\tilde{f}(x))(i)=f(x(i))$, and satisfying the
conditions of Lemma~\ref{FilHom}. 
\end{proof}

\begin{lemma}\label{lemmaHSPkites}
Let $\mathbf B$ be a B-cycle, and let $\mathbf{L}_1$, $\mathbf{L}_2$,
$\mathbf L_j$ \textup{(}$j\in J$\textup{)}, be $\ell$-groups. The following hold.
\begin{align}
\mathbf L_1\in S(\mathbf L_2)  \quad&\Rightarrow\quad 
\mathcal K_\mathbf B(\mathbf L_1)\in S(\mathcal K_\mathbf B(\mathbf L_2))\tag{i} \\
  \mathbf L_1\in H(\mathbf L_2)  \quad&\Rightarrow\quad
\mathcal K_\mathbf B(\mathbf L_1)\in H(\mathcal K_\mathbf B(\mathbf L_1))\tag{ii}\\ 
\mathcal K_\mathbf B(\prod_{j\in J}\mathbf L_j) &=
\pprod_{j\in J}  \mathcal{K}_\mathbf B(\mathbf L_j)\tag{iii}
\end{align}
\end{lemma}

\begin{proof}
Note that injective (surjective) $f$ in Lemma~\ref{KiteHoms} induces
an injective (surjective) $\tilde{f}$. From this, (i) and (ii) follow immediately.
To demonstrate (iii), for any $x\in\prod_{j\in J}((L^-_j)^B\uplus(L^+_j)^B)$
we have
$$
x\in \pprod_{j\in J} \bigl((L^-_j)^B\uplus(L^+_j)^B\bigr)
\qquad\text{if and only if}\qquad
x\in\prod_{j\in J}(L^-_j)^B\quad\text{or}\quad x\in\prod_{j\in J}(L^+_j)^B
$$
by definition. Hence
$$
\pprod_{j\in J} \bigl((L^-_j)^B\uplus(L^+_j)^B\bigr)= \prod_{j\in
  J}(L^-_j)^B\uplus \prod_{j\in J}(L^+_j)^B
$$
as required.
\end{proof}

\begin{cor}\label{cor-HSPkites}
Let $\mathbf B$ be a B-cycle, and let $\mathcal{V}$ be a variety of 
pseudo MV-algebras. Then the classes
$\{\mathbf L\in\var{LG}: \mathcal{K}_{\mathbf{B}}(\mathbf L)\in \mathcal{V}\}$
and $\{\mathbf L^-\in\var{\cigmv}: \mathcal{K}_{\mathbf{B}}(\mathbf L)\in \mathcal{V}\}$
are varieties.   
\end{cor}

\section{Approximation of perfect pseudo MV-algebras by kites}\label{sec:approximation}

Let $\mathbf{L}$ be an $\ell$-group, and let
$f\colon \mathbf B\rightarrow \mathbf C$ be a homomorphism of B-cycles.
Then $f$ naturally lifts to a homomorphism 
$$
\mathcal K_ f (\mathbf L)\colon \mathcal K_\mathbf C (\mathbf L)\rightarrow
\mathcal K_\mathbf B (\mathbf L)
$$
defined by $(\mathcal{K}_f(\mathbf L))(x)=x\circ f$.
Moreover
$$
\mathcal K_-(\mathbf L)\colon \var{BC}\rightarrow \var{\pfpmv}
$$
is a contravariant functor. Since 
$\sn\sn$ is an automorphism on the filters of perfect pseudo MV-algebras,
for any $\mathbf{A}\in \var{\pfpmv}$ we have that
$(F_\mathbf{A}; {\sn\sn})$ is a B-cycle. Therefore
we can correctly state the following definition.  

\begin{defin}\label{spec-cyc}
For a perfect MV-algebra $\mathbf A$ and an $\ell$-group $\mathbf L$ we
define a B-cycle $[\mathbf A,\mathbf L]\dfeq ([A,L]; \lambda)$, where
$$
[A,L] \dfeq \{\alpha\in (L^-)^{F_\mathbf{A}}\colon
\alpha \text{ is a homomorphism in }\var\cigmv\}  
$$
given by $\lambda(\alpha (x))\dfeq \alpha (x\sn\sn)$
for any $\alpha \in [A,L]$. 
\end{defin}

\begin{lemma}\label{K-contravariant}
If $f\colon \mathbf A_1\rightarrow \mathbf A_2$ is a homomorphism of perfect
pseudo MV-algebras, then 
$$
[f,\mathbf L] \colon [ \mathbf A_2,\mathbf L ]\rightarrow [ \mathbf
A_1,\mathbf L ]
$$
given by $[ f,\mathbf L](\alpha)=\alpha\circ f$, is a homomorphism of B-cycles.
Moreover
$$
[-,\mathbf L]\colon  \var{\pfpmv}\rightarrow \var{BC}
$$ 
is a contravariant functor.
\end{lemma}

\begin{proof}
Immediate by Definition~\ref{spec-cyc}.
\end{proof}  

\begin{theorem}
Let $\mathbf{L}$ be an $\ell$-group. Then
\begin{enumerate}
\item For any perfect pseudo MV-algebra $\mathbf A$ there exists a homomorphism
$$
\eta_\mathbf A \colon \mathbf A\rightarrow \mathcal
K_{[\mathbf A,\mathbf L]}(\mathbf L)
$$
defined by 
$$
(\eta_\mathbf A (x))(\alpha)= \begin{cases}
\alpha (x) & \text{ if }  x\in F_\mathbf A,\\
\alpha(x\sn)^{-1} & \text{ if } x\in J_\mathbf A.
\end{cases}
$$
for any $x\in A$ and any $\alpha \in [A,L]$. Then
$$
\eta = (\eta_\mathbf A \colon \mathbf A\rightarrow
\mathcal K_{[\mathbf A,\mathbf{L}]}(\mathbf{L}))_{\mathbf{A}\in |\var{P\Psi{}MV}|}
$$
is a natural transformation
from the identity endofunctor to the endofunctor
$K_{[\mathbf -,\mathbf{L}]}(\mathbf L)$. 

\item For any B-cycle $\mathbf B = (B;\lambda)$ there exists a homomorphism
$$
\varepsilon_\mathbf B\colon \mathbf B\rightarrow [\mathcal{K}_\mathbf
B(\mathbf L),\mathbf L]
$$
such that $(\varepsilon_\mathbf B(i))(x)=x(i)$ for any $i\in B$
and $x\in (L^-)^B$, hence $\varepsilon_\mathbf B(i)=\pi_i\colon
(L^-)^B\rightarrow L^-$. Then
$$
\varepsilon = (\varepsilon_\mathbf B\colon \mathbf B\rightarrow
[\mathcal{K}_\mathbf B(\mathbf L),\mathbf L])_{\mathbf B\in |\var{BC}|}
$$
is a natural transformation from the identity endofunctor to the endofunctor
$[\mathcal{K}_\mathbf -(\mathbf L),\mathbf L]$.

\item The pair $(\eta,\varepsilon)$ forms an adjunction between the 
functors $[-,\mathbf L]$ and $\mathcal K_-(\mathbf L)$.
\end{enumerate}  
\end{theorem}

\begin{proof}
For (1) we first prove that $\eta_\mathbf A$ is a homomorphism. It
is clear that the mapping 
$$
n_\mathbf A \colon \mathbf F_\mathbf A\rightarrow (\mathbf L^-)^{[A,L]} =
\mathbf F_{\mathcal K_{[\mathbf A,\mathbf   L]}(\mathbf L)}
$$
defined by $(n_\mathbf A (x))(\alpha)=\alpha (x)$ for any $x\in A$ and
$\alpha\in [A,L]$ is a homomorphism of IGMV-algebras (because $\alpha \in [A,L]$
are homomorphisms of IGMV-algebras). Moreover $n_\mathbf A$ satisfies 
$$
(n_\mathbf A (x\sn\sn))(\alpha)=\alpha (x\sn\sn)=(\lambda\circ\alpha)(x)=(n_\mathbf A
(x))(\lambda\circ\alpha)=(n_\mathbf A (x)\sn\sn)(\alpha)
$$
for any $x\in F_\mathbf A$ and $\alpha\in [A,L]$.
Then using Lemma~\ref{FilHom}, and adopting the over-bar notation from there,
we get that $\eta_\mathbf A = \overline{n}_\mathbf A$ is a homomorphism. 

 To prove that $\eta$ is a natural transformation, let  $f\colon \mathbf
 A\rightarrow\mathbf B$ be a homomorphism of perfect pseudo MV-algebras. We
 show that the diagram: 
\begin{center}
\begin{tikzcd}[column sep=large,row sep=large]
\mathbf{A} \arrow[r, "f"] \arrow[d, "\eta_\mathbf{A}"'] &
\mathbf{B} \arrow[d, "\eta_\mathbf{B}"] \\
\mathcal{K}_{[\mathbf{A},\mathbf{L}]}(\mathbf{L})
\arrow[r, "\mathcal K_{[f,\mathbf L]}(\mathbf{L})"] &
\mathcal{K}_{[\mathbf{B},\mathbf{L}]}(\mathbf{L})
\end{tikzcd}
\end{center}
commutes. Let $\alpha\in [\mathbf A,\mathbf L]$. For $x\in J_\mathbf A$,  we have
\begin{align*}
((\mathcal K_{[f,\mathbf L]}(\mathbf L)\circ \eta_\mathbf A)(x))(\alpha) &=
   (\eta_\mathbf A(x)\circ [f,\mathbf L])(\alpha) \\
&=  (\eta_\mathbf A(x))(\alpha\circ f) = \alpha (f(x\sn))^{-1}\\
&=  \alpha (f(x)\sn)^{-1}\\
&=  ( (\eta_\mathbf B\circ f)(x))(\alpha).
\end{align*} 
For $x\in F_\mathbf A$, we have
\begin{align*}
  ((\mathcal K_{[f,\mathbf L]}(\mathbf L)\circ \eta_\mathbf A)(x))(\alpha)
&=  (\eta_\mathbf A(x)\circ [f,\mathbf L])(\alpha) \\
&= (\eta_\mathbf A(x))(\alpha\circ f) = \alpha (f(x))\\
&= ( (\eta_\mathbf B\circ f)(x))(\alpha).
\end{align*}
Consequently $\eta$ is a natural transformation, as claimed.

For (2), let $\mathbf B=(B;\lambda)$ be a B-cycle. It is clear that
$\varepsilon_\mathbf B(i)$ is the projection homomorphism and hence
$\varepsilon_\mathbf B\in [\mathcal K_\mathbf B(\mathbf L),\mathbf L]$ for any
$i\in B.$ Since the equality 
$$\lambda \varepsilon_\mathbf B (i)(x)= \varepsilon_\mathbf B
(i)(x\sn\sn)=(x\sn\sn)(i)=x(\lambda i)= (\varepsilon_\mathbf B (\lambda i))(x)$$
holds for all $i\in B$ and $x\in (L^-)^B$, we get that $\varepsilon_\mathbf B$ is a
homomorphism of B-cycles.  
Next, for a homomorphism $f\colon \mathbf B\rightarrow\mathbf C$ of B-cycles,
we verify that the diagram
\begin{center}
\begin{tikzcd}[column sep=large,row sep=large]
\mathbf{B} \arrow[r, "f"] \arrow[d, "\varepsilon_\mathbf{B}"'] &
\mathbf{C} \arrow[d, "\varepsilon_\mathbf{C}"] \\
\lbrack\mathcal{K}_{\mathbf{B}}(\mathbf{L}),\mathbf{L}\rbrack
 \arrow[r, "\lbrack\mathcal{K}_{f}(\mathbf{L}){,} \mathbf{L}\rbrack"] &
\lbrack\mathcal{K}_{\mathbf{C}}(\mathbf{L}),\mathbf{L}\rbrack
\end{tikzcd}
\end{center}
commutes. For any $i\in F$ and any $x\in (L^+)^C$ we have
\begin{align*}
(([\mathcal K_f(\mathbf L),\mathbf L]\circ \varepsilon_\mathbf B)(i))(x)
&= (\varepsilon_\mathbf B (i))((\mathcal K_\mathbf f (\mathbf L))(x))\\
&= ((\mathcal K_\mathbf f (\mathbf L))(x))(i)\\
&=  x(f(i))\\
&= ((\varepsilon_\mathbf C\circ f)(i))(x).
\end{align*}
Hence, $\varepsilon$ is a natural transformation.

Finally, for (3) we need to verify commutativity of two diagrams below. 
\begin{center}
\begin{tikzcd}[column sep=large,row sep=large]
\mathcal{K}_\mathbf{B}(\mathbf{L})
\arrow[r, "\eta_{\mathcal{K}_\mathbf{B}(\mathbf{L})}"]
\arrow[rd, equals] &
\mathcal{K}_{\lbrack\mathcal{K}_\mathbf{B}(\mathbf{L}),\mathbf{L}\rbrack}(\mathbf L)
\arrow[d, "\mathcal{K}_{\varepsilon_\mathbf{B}}(\mathbf L)"] \\
  & \mathcal{K}_\mathbf{B}(\mathbf{L})
\end{tikzcd}
\qquad
\begin{tikzcd}[column sep=large,row sep=large]
\lbrack\mathbf A,\mathbf L\rbrack
\arrow[r, "\varepsilon_{\lbrack\mathbf A{,}\mathbf L\rbrack}"]
\arrow[rd, equals] &
\lbrack\mathcal{K}_{\lbrack\mathbf{A},\mathbf{L}\rbrack}(\mathbf{L}),\mathbf{L}\rbrack
\arrow[d, "\lbrack \eta_\mathbf A{,}\mathbf L\rbrack"] \\
  & \lbrack\mathbf A,\mathbf L\rbrack
\end{tikzcd}
\end{center}
Take $x\in \mathcal{K}_\mathbf{B}(\mathbf{L})$ and $i\in B$; we have two cases. 
If $x\in F_{\mathcal K_\mathbf B (\mathbf L)}=(L^-)^B$, then
$$
((\mathcal K_{\varepsilon_\mathbf B}(\mathbf L)\circ \eta_{\mathcal K_f(\mathbf
  L)})(x))(i) =( \eta_{\mathcal K_f(\mathbf L)}(x))(\varepsilon_\mathbf B (i)) =
(\varepsilon_\mathbf B (i))(x) = x(i).  
$$
If $x\in J_{\mathcal K_\mathbf B (\mathbf L)}=(L^+)^B$, then
$$
((\mathcal K_{\varepsilon_\mathbf B}(\mathbf L)\circ \eta_{\mathcal K_f(\mathbf L)})(x))(i) =( \eta_{\mathcal K_f(\mathbf L)}(x))(\varepsilon_\mathbf B (i)) = (\varepsilon_\mathbf B (i))(x\sn)^{-1} = x(i)^{-1-1}=x(i). 
$$
So the left diagram commutes. For the right diagram, taking
$\alpha\in [\mathbf A,\mathbf L]$ and $x\in J_\mathbf A$, we have
$$
(([\eta_\mathbf A,\mathbf L]\circ \varepsilon_{[\mathbf A,\mathbf
  L]})(\alpha))(x)=( \varepsilon_{[\mathbf A,\mathbf L]}(\alpha))(\eta_\mathbf
A(x))=(\eta_\mathbf A(x))(\alpha)=\alpha (x),
$$
so it commutes, too.
\end{proof}

\begin{cor}
Let $\mathbf A$ be a perfect pseudo MV-algebra, $\mathbf B$ be a B-cycle and
$\mathbf L$ be an $\ell$-group. Then for any homomorphism $f\colon \mathbf
A\rightarrow\mathcal K_\mathbf B(\mathbf L)$ there exists a unique
homomorphism of B-cycles $g\colon \mathbf B\rightarrow [\mathbf A,\mathbf L]$
such that the diagram 
\begin{center}
\begin{tikzcd}[column sep=large,row sep=large]
\mathbf{A} \arrow[r, "\eta_\mathbf{A}"] \arrow[rd,"f"'] &
\mathcal{K}_{\lbrack \mathbf{A},\mathbf{L}\rbrack}(\mathbf{L})
\arrow[d, "\mathcal{K}_g(\mathbf{L})"] \\
  & \mathcal{K}_\mathbf{B}(\mathbf{L})
\end{tikzcd}
\end{center}
commutes. Then $g=[f,\mathbf L]\circ \varepsilon_\mathbf B$ holds.
\end{cor} 
Hence the B-cycle $[\mathbf A,\mathbf L]$ is in some sense the best B-cycle to
approximate the pseudo MV-algebra $\mathbf A$ by the kite construed by given
$\ell$-group $\mathbf L$.  

\section{Varieties generated by kites}\label{sec:varieties}

Recall that $\Lambda(\mathcal{V})$ stands for the lattice of
subvarieties of $\mathcal{V}$, and $\Lambda^+(\mathcal{V})$ for 
the poset of nontrivial subvarieties of $\mathcal{V}$. Note however that
$\Lambda^+(\var{\ppmv})$ is a lattice, indeed, a complete algebraic sublattice
of $\Lambda(\var{\ppmv})$; the bottom element of $\Lambda^+(\var{\ppmv})$
is the variety $\var{BA}$ of
Boolean algebras, which is the unique atom of $\Lambda(\var{\ppmv})$.  

Throughout the section $\mathbb{D}$ will stand for the lattice 
$(\mathbb{N};\mid)$ of natural numbers under 
the divisibility ordering (with $0$ the top element).
All order theoretic notions: minima, maxima, suprema,
etc., will be taken with respect to this ordering, unless clearly indicated
otherwise. For any bijection $\lambda$ on a nonempty set $B$,
we put, inductively, $\lambda^0(x) \dfeq x$ and
$\lambda^{i+1}(x) \dfeq \lambda(\lambda^i(x))$; then we define the
\emph{dimension} of $\lambda$ as follows:
$$
\ord(\lambda)\dfeq \textstyle{\min^{\mathbb{D}}}\{n\in\mathbb{N}:\lambda^n = id_B\}.
$$
Note that the minimum is well defined since $\mathbb{D}$ satisfies the
descending chain condition. For a B-cycle $\mathbf B=(B;\lambda)$, we put
$\ord(\mathbf{B})\dfeq\ord(\lambda)$ and call it the dimension of $\mathbf{B}$.
In particular, if $\lambda = id_B$, then $\ord(\mathbf{B}) = 1$,
and if $\lambda^n\neq id_B$ for all non-zero $n$, then $\ord(\mathbf{B}) = 0$.

In this section (but only in this section), we will abbreviate the term
operation $(-\sn\sn)$ on a pseudo MV-algebra by $(-^\approx)$, and put:
$$
(-^\approx)=(-^{1\times\approx}) \dfeq
(-^{\sn\sn}),\,(-^{(n+1)\times\approx}) \dfeq (-^{n\times\approx})^\approx.
$$
For any pseudo MV-algebra $\mathbf{A}$,
the operation $-^\approx$ is a bijection on $A$, so for any
$\mathbf{A}$ we define the dimension of $\mathbf{A}$
to be $\ord(-^\approx)$. This is essential
for the rest of the section, so we state it formally.

\begin{defin}\label{psMV-dimension}
Let $\mathbf{A}\in\var{\Psi MV}$ and $\avar{V}\in\Lambda(\var{\Psi MV})$.
Then
\begin{enumerate}
\item $\ord(\mathbf{A})\dfeq \ord(-^\approx)$,
\item $\ord(\avar V) \dfeq 
\min^{\mathbb{D}}\{n\in\mathbb{D}: \ord(\mathbf{A})\mid n \text{ for all }\mathbf{A}\in
\avar{V}\}$,
\item $\var{\ppmv}_n \dfeq \var{\ppmv}\cap\mathrm{Mod}\{\lambda^n(x) = x\}$, for
  any $n\in \mathbb{D}$.
\end{enumerate}
\end{defin}
It is immediate that $\var{\ppmv}_n$ defined in (3) 
is the largest subvariety of $\var{\ppmv}$ of dimension $n$. Moreover,
for all $n,m\in\mathbb{N}$ we have
$$
\var\ppmv_n\subseteq\var\ppmv_m \text{ if and only if } n\mid m
$$
so in particular $\var{\ppmv}_0 = \var{\ppmv}$. 
All finite subdirectly irreducible cycles are of the form
$\mathbf{Z}_n\dfeq (\{0,1,\dots,n-1\};\lambda_n)$, where
$\lambda_n(m)\dfeq m+1\pmod n$ for any $m\in Z_n$. To spare notation, 
for any $\ell$-group $\mathbf{L}\in \var{LG}$
we will write $\mathcal K_n(\mathbf L)$ to denote the kite
$\mathcal{K}_{\mathbf{Z}_n}(\mathbf L)$. This agrees with the terminology
of Dvure\v{c}enskij, Kowalski~\cite{DK14}, except that dimension 0 here is
infinite dimension in~\cite{DK14}. The results from 
Section 6 of~\cite{DK14} do not state it explicitly, but their proofs show
that $\mathcal K_0(\mathbf L)$ subdirectly embeds into
$\prod_{i\in\mathbb{D}}\mathcal K_i(\mathbf L)$, so for variety generation only
finite-dimensional kites matter.  

The lattice of subvarieties of $\var{BC}$ has a unique atom, namely the variety
$\var{S}$ consisting of all algebras $\mathbf{S} = (S; id_S)$.
Obviously $\var{S}$ is term equivalent to the variety of bare sets,
and any set can be viewed as a B-cycle from $\var{S}$, so 
we introduce an \emph{ad hoc} notation
$\mathbf{B}\times S$ for $\mathbf{B}\times\mathbf{S}$ with
$\mathbf{S}\in\var{S}$.

\begin{lemma}\label{powerlemma}
Let $\mathbf B$ be a B-cycle, $\mathbf L$ be a $\ell$-group and $S$ a set. Then
$$
\mathcal K_{\mathbf{B}\times S}(\mathbf L)\cong
\pprod_S \mathcal K_\mathbf B(\mathbf L).
$$
\end{lemma}

\begin{proof}
Putting $\mathbf{D} = \pprod_S \mathcal K_\mathbf B(\mathbf L)$, we have  
$$
F_\mathbf{D} = F_{\pprod_S \mathcal K_\mathbf B(\mathbf L)} =
\bigl((\mathbf L^S)^-\bigr)^B\cong
\bigl((\mathbf L^-)^S\bigr)^B\cong (\mathbf L^-)^{B\times S} =
F_{\mathcal K_{\mathbf{B}\times S}(\mathbf L)}.
$$
It is straightforward to check that the natural isomorphisms
above preserve automorphisms. Then the claim follows by 
Lemma~\ref{FilHom}. 
\end{proof}

\begin{lemma}\label{cycleISP}
Let $\mathbf L$ be an $\ell$-group and let $\mathbf B$ be a B-cycle. Then
$$
\mathcal K_\mathbf B(\mathbf L)\in {ISP}(\mathcal K_{\ord(\mathbf B)}(\mathbf L)).
$$
\end{lemma}

\begin{proof}
Assume $\ord(\mathbf B) = n$. The map
$$
f \colon \mathbf{Z}_n\times B\rightarrow \mathbf B
$$
defined by $f(m,i)=\lambda^m(i)$ for any 
$m\in\{0,1,\dots, n-1\}$ and any $i\in B$ is a homomorphism. 
Applying the contravariant functor $\mathcal{K}_f(-)$ we get that
$$
\mathcal K_f(\mathbf L)\colon
\mathcal K_\mathbf B (\mathbf L)\rightarrow \mathcal K_{\mathbf{Z}_n\times B}(\mathbf L)
$$  
is a homomorphism as well. 
Suppose 
$(\mathcal K_f(\mathbf L))(x)=(\mathcal K_f(\mathbf L))(y)$ for some
$x,y\in F_{\mathcal K_\mathbf B (\mathbf L)}=(L^-)^B$. Then 
$$
x(i)=x(f(i,0))=(\mathcal K_f(\mathbf L))(x)=
(\mathcal K_f(\mathbf L))(y)=y(f(i,0))=y(i)
$$
for all $i\in B$ and thus $x=y$;  hence $f$ is an embedding.  
So 
$\mathcal K_{\mathbf Z_n\times B}(\mathbf L)\in ISP (\mathcal K_{n}(\mathbf L))$
by Lemma \ref{powerlemma}.
\end{proof} 

\begin{defin}\label{Galois-maps}
For any $\mathcal{V}\in\Lambda(\var{\ppmv})$  we put:
\begin{align*}
\psi&\colon \Lambda(\var{\ppmv}) \rightarrow  \Lambda (\var{\cigmv}),
\text{ where }
\psi (\mathcal{V}) = V\{\mathbf F_\mathbf A: \mathbf A\in\mathcal{V}_{\pf}\},\\
\Psi&\colon \Lambda(\var{\ppmv}) \rightarrow  \Lambda (\var{\cigmv})\times
\mathbb{D},
\text{ where }
\Psi (\mathcal{V}) = (\psi (\mathcal{V}),\, \ord(\mathcal{V})).
\end{align*}
Further, for any $\mathcal{V}\in\Lambda(\var {\cigmv})$ and $n\in\mathbb{D}$, we
put: 
\begin{align*}
\delta&\colon \Lambda (\var{\cigmv}) \rightarrow  \Lambda(\var{\ppmv}),
\text{ where }
\delta (\mathcal{V})= V\{\mathbf A\in
\var {\pfpmv}: \mathbf{F}_{\mathbf{A}}\in\mathcal{V}\},\\
\Delta&\colon \Lambda (\var{\cigmv})\times
\mathbb{D} \rightarrow  \Lambda(\var{\ppmv}),
\text{ where }
\Delta (\mathcal{V},n)=\delta (\mathcal{V})\cap \var {\ppmv}_n.
\end{align*}

\end{defin}

Several little facts follow immediately from the definitions. We gather
them in the next lemma. 

\begin{lemma}\label{little-facts}
Let $\psi$, $\delta$, $\Psi$ and $\Delta$ be as above. The following hold.
\begin{enumerate}
\item The maps $\psi$, $\delta$, $\Psi$ and $\Delta$ are
  monotone.\label{little-facts-1}
\item $\psi(\var{Tr}) = \psi(\var{BA}) = \var{Tr}$,
 and $\Psi(\var{Tr}) = \Psi(\var{BA}) = (\var{Tr}, 1)$.\label{little-facts-2}
\item $\psi(\var{\ppmv}) = \var{\cigmv}$, and $\Psi(\var{\ppmv}) =
  (\var{\cigmv}, 0)$.\label{little-facts-3} 
\item $\delta(\var{Tr}) = \var{BA}$, and $\Delta(\var{Tr}, n) = \var{BA}$ for
  all $n\in\mathbb{N}$.\label{little-facts-4}
\item $\delta(\var{\cigmv}) = \var{\ppmv}$,  and $\Delta(\var{\cigmv}, n) =
  \var{\ppmv}_n$, for any $n\in\mathbb{N}$.\label{little-facts-5}
\end{enumerate}
\end{lemma}  

The next lemma gives equational bases for varieties $\delta(\avar{V})$, relative to
$\avar{V}$. Recall that any equation $\varepsilon$ in the language of residuated
lattices is equivalent, over residuated lattices, to an equation
of the form $t_\varepsilon = 1$ for a term $t_\varepsilon$ effectively obtainable
from $\varepsilon$. Thus, for any set $E$ of equations we have  
$$
\avar{V}\models E \quad\Longleftrightarrow\quad
\avar{V}\models \{t_\varepsilon = 1: \varepsilon\in E\}
$$
and, in particular, $E$ is an equational base for $\avar{V}$ if and only if 
$\{t_\varepsilon = 1: \varepsilon\in E\}$ is.

\begin{lemma}\label{CharIden}
Let $\avar{V}\in\Lambda (\var{\cigmv})$, let $E$ be an equational base for
$\avar{V}$,  and let $\mathbf A\in \var {\ppmv}$. The following are equivalent.
\begin{enumerate}
\item $\mathbf{A}\in \delta(\avar{V})$,
\item $\mathbf{A} \models t(x_1\vee x_1^-,\dots,x_k\vee x_k^-) = 1$
for all terms $t$ in the language of residuated lattices, such that 
$\avar{V}\models t(x_1,\dots,x_k) = 1$.
\item $\mathbf{A} \models t_\varepsilon(x_1\vee x_1^-,\dots,x_k\vee x_k^-) = 1$
for all equations $\varepsilon(x_1,\dots,x_k)\in E$.
\end{enumerate}
If $\mathbf{A}$ is perfect, then the conditions above are equivalent to
the fact that $\mathbf{F}_{\mathbf{A}}\in\avar{V}$.
\end{lemma}

\begin{proof}
For (1) $\Rightarrow$ (2)
let $t$ be a $k$-ary term in the language of residuated lattices such that  
$\avar{V}\models t(x_1,\dots,x_k)= 1$. We can assume without loss that
$\mathbf{A}$ is among the generators of $\delta(\avar{V})$.
Then $\mathbf F_\mathbf{A}\in \mathcal{V}$, so for any
$a_1,\dots,a_k\in A$ we have $a_i\vee a_i^-\in F_\mathbf A$ for each
$i\in\{1,\dots, k\}$. Hence 
$$
t^\mathbf{A}(a_1\vee a_1^-,\dots,a_k\vee a_k^-)=
t^{\mathbf{F}_\mathbf{A}}(a_1\vee a_1^-,\dots,a_k\vee a_k^-) =1 
$$
and consequently  
$$
\mathbf A \models t(x_1\vee x_1^-,\dots,x_k\vee x_k^-)= 1.
$$ 

For (2) $\Rightarrow$ (3) note that $\varepsilon\in E$ implies
$\avar{V}\models t_\varepsilon = 1$. Assume
$t_\varepsilon = t_\varepsilon(x_1,\dots,x_k)$. Then
$\mathbf{A}\models t_\varepsilon(x_1\vee x_1^-,\dots,x_k\vee x_k^-) = 1$, by (2).

For (3) $\Rightarrow$ (1), suppose $\mathbf{A}$ satisfies the required equations. 
Let $(\mathbf{A}_i)_{i\in I}$ be a subdirect decomposition of
$\mathbf{A}$ with subdirectly irreducible factors. By Lemma~\ref{si-perfect},
each $\mathbf{A}_i$ is perfect. Hence, it suffices to show that
$\mathbf{F}_{\mathbf{A}_i}\in\mathcal{V}$, for each $i\in I$.
Take any $\mathbf{A}_i$ and any equation $\varepsilon(x_1,\dots, x_k)\in E$, so
that $\avar{V}\models {t_\varepsilon(x_1,\dots, x_k) = 1}$.
Next, take any $a_1,\dots, a_k\in F_{\mathbf{A}_i}$. Then
$a_j=a_j\vee a_j^-$ holds for all $j\in \{1,\dots k\}$ and so we have
$$
1 = t^{\mathbf{A}_i}(a_1\vee a_1^-,\dots, a_k\vee a_k^-)
= t^{\mathbf F_{\mathbf{A}_i}}(a_1,\dots,a_k) 
$$
where the first equality holds by (3). It follows that
$\mathbf{F}_{\mathbf{A}_i}\in\avar{V}$, as needed.

The last statement is immediate by the definition of $\delta(\avar{V})$.
\end{proof}

\begin{lemma}\label{psigamma}
The equality $\mathcal{V}= \psi\delta (\mathcal{V})$ holds
for any $\mathcal{V}\in \Lambda (\var{\cigmv})$.
\end{lemma}

\begin{proof}
Let $\mathbf A\in\mathcal{V}\in\Lambda (\var{\cigmv})$, and 
let $\mathbf{B} = \mathcal K(\ell(\mathbf A),1_{\ell(A)})$. Since
$\mathbf{F}_{\mathbf{B}} = \ell(\mathbf{A})^-=\mathbf{A}$ we get
$\mathcal K(\ell(\mathbf{A}),1_{\ell(A)})\in\delta (\mathcal{V})$, and therefore
$\mathbf{A}\in\psi\delta (\mathcal{V})$.
Hence $\mathcal{V}\subseteq \psi\delta (\mathcal{V})$.

To show the opposite inclusion, take any $\mathcal{V}\in\Lambda(\var{\cigmv})$.
We have $\psi\delta (\mathcal{V}) = V\{\mathbf{F}_{\mathbf{B}}:
\mathbf{B}\in\delta(\mathcal{V})_{\pf}\}$, so it suffices to show
that
\begin{equation}\label{models}\tag{\dag}
\{\mathbf{F}_{\mathbf{B}}:\mathbf{B}\in\delta(\mathcal{V})_{\pf}\}
\models t(x_1,\dots,x_k)=1
\end{equation}
holds for every identity $t(x_1,\dots,x_k)=1$ satisfied by $\mathcal{V}$.
Let $t(x_1,\dots,x_k)=1$ be such an identity. By Lemma~\ref{CharIden} we obtain
$$
\delta (\mathcal{V}) \models t(x_1\vee x_1^-,\dots,x_k\vee x_k^-)=1.
$$
Now take any $\mathbf{B}\in\delta(\avar{V})_{\pf}$ and
any $b_1,\dots,b_k\in F_\mathbf{B}$. Since $b_i = b_i\vee b_i^-$ holds for any
$i\in\{1,\dots,k\}$ we have 
$$
t^{\mathbf{F}_\mathbf{B}}(b_1,\dots, b_k) =
t^\mathbf{B}(b_1\vee b_1^-,\dots, b_k\vee b_k^-)=1
$$
and therefore
$$
\mathbf F_\mathbf{B}\models t(x_1,\dots,x_k)=1.
$$
As $\mathbf{B}$ was arbitrarily chosen, \eqref{models} holds.
Hence $\psi\delta (\mathcal{V})\subseteq \mathcal{V}$, and so
$\psi\delta (\mathcal{V}) = \mathcal{V}$ as claimed.
\end{proof}

\begin{lemma}\label{gammapsi}
For any $\mathcal{V}\in\Lambda (\var{\ppmv})$ we have
$\mathcal{V}\leq \delta\psi (\mathcal{V})$. 
\end{lemma}

\begin{proof}
If $\mathbf A\in\mathcal{V}_{\pf}$, then
$\mathbf F_\mathbf A\in \psi(\mathcal{V})$ holds
and so we have $\mathbf A\in\delta\psi
(\mathcal{V})$. Since $\mathcal{V}$ is  perfectly generated, the claim follows. 
\end{proof}

\begin{lemma}\label{KitesLemma}
Let $\mathcal{M}\subseteq \var{\pfpmv}$ be an arbitrary class.
If $\mathbf{A}\in V(\mathcal{M})$ is perfect, then
$$
\mathbf{F}_\mathbf{A}\in V\{\mathbf{F}_\mathbf{B}: \mathbf{B}\in \mathcal{M}\}.
$$
\end{lemma}

\begin{proof}
Let $\mathbf A\in V(\mathcal{M})$ be a perfect pseudo MV-algebra. Then there exist
algebras $\mathbf A_i\in \mathcal{M}$ ($i\in I$), such that for some
$\mathbf B\leq\prod_{i\in I} \mathbf A_i$ and some
$\theta\in \mathrm{Con}(\mathbf B)$ we have $\mathbf A\cong\mathbf B/\theta$.
Let
$$
C\dfeq  B\cap \pprod _{i\in I} A_i.
$$
Then $C$ is the universe of a subalgebra $\mathbf{C}$ of $\mathbf{B}$.
Moreover, $\mathbf{C}$ is perfect, since $\mathbf{B}$ is.
Let $\theta_\mathbf C=\theta\cap C^2$. The map 
$$
f\colon \mathbf C/\theta_\mathbf C\rightarrow\mathbf{B}/\theta
$$
defined by putting $f(x/\theta_\mathbf C) = x/\theta$ for any $x\in C$ is clearly a
homomorphism. Suppose $f(x/\theta_\mathbf C)=f(y/\theta_\mathbf C)$ for some
$x,y\in C$. Then, 
$\langle x,y\rangle \in \theta\cap C^2=\theta_\mathbf{C}$
which proves injectivity of $f$.
To show surjectivity, take a
$z\in F_\mathbf{A} = F_{\mathbf B/\theta}$. Then for some 
$x\in B$ we have $x/\theta = z$ and therefore 
$x/\theta =(x\vee x^-)/\theta\in F_{\mathbf{B}/\theta}$. Hence,
$(x\vee x^-)(i)\in F_{\mathbf{A}_i}$ for any $i\in I$ and so
$x\vee x^-\in \pprod_{i\in I}\mathbf A_i$. 
It follows that
$$
f((x\vee x^-)/\theta_\mathbf C)=(x\vee x^-)/\theta= x/\theta = z
$$
and by Lemma~\ref{FilHom} we get that $f$ is surjective. This shows
that $\mathbf{A}\cong \mathbf C/\theta_\mathbf C$, and therefore
$$
\mathbf F_\mathbf A\cong \mathbf F_{\mathbf C/\theta_\mathbf C}\in H(\mathbf
F_\mathbf C)\subseteq HSP\{\mathbf A_i: i\in I\}
$$
as needed. 
\end{proof}

\begin{lemma}\label{kitegeneratedlemma}
If a variety $\mathcal{V}\in\Lambda (\var{\ppmv})$ is generated by kites, then 
$$
\mathbf L^-\in\psi (\mathcal{V})\quad\text{ if and only if }\quad
\mathcal K_{\ord(\mathcal{V})}(\mathbf L)\in\mathcal{V}
$$
holds for any $\ell$-group $\mathbf L$.
\end{lemma}

\begin{proof}
Let $\ord(\mathcal{V}) = n$. Take an $\ell$-group $\mathbf{L}$ such that
$\mathcal K_{n}(\mathbf L)\in \mathcal{V}$.
Then $\mathbf F_{\mathcal{K}_{n}(\mathbf{L})} = (\mathbf L^-)^{n}\in\psi (\mathcal{V})$
and $\mathbf L^-\in H((\mathbf{L}^-)^{n})$ via a projection map.  
Since $\psi(\mathcal{V})$ is a variety, we get $\mathbf L^-\in \psi (\mathcal{V})$.

For converse, let $\mathcal{K} \dfeq \{\mathcal K_{n}(\mathbf L)\in\mathcal{V}:
\mathbf{L}\in\var{LG}\}$.
Because $\mathcal{V}$ is generated by kites, by Lemma~\ref{cycleISP}
we get that $\mathcal{V}= V(\mathcal{K})$. Therefore,
\begin{equation}\label{one}\tag{i}
\psi(\mathcal{V}) = V\{\mathbf{F}_\mathbf{A}: \mathbf{A}\in\mathcal{V}_{\pf}\}
= V\{\mathbf{F}_\mathbf{A}: \mathbf{A}\in (V(\mathcal{K}))_{\pf}\}.
\end{equation}
Taking $\mathcal{K}$ as $\mathcal{M}$ in Lemma \ref{KitesLemma}, we get
$\{\mathbf{F}_\mathbf{A}: \mathbf{A}\in (V(\mathcal{K}))_{\pf}\}
\subseteq V\{\mathbf{F}_\mathbf A: \mathbf{A}\in \mathcal{K}\}$ and hence
\begin{equation}\label{two}\tag{ii}
V\{\mathbf{F}_\mathbf{A}: \mathbf{A}\in (V(\mathcal{K}))_{\pf}\}
\subseteq VV\{\mathbf{F}_\mathbf A: \mathbf{A}\in \mathcal{K}\} =
V\{\mathbf{F}_\mathbf A: \mathbf{A}\in \mathcal{K}\}.
\end{equation}
Further, unwinding the definition of $\mathcal{K}$ we obtain 
\begin{align}\label{three}\tag{iii}
\begin{split}
V\{\mathbf{F}_\mathbf A:\mathbf{A}\in \mathcal{K}\} & =
V\{\mathbf{F}_{\mathcal K_{n}(\mathbf L)}: \mathcal{K}_{n}(\mathbf L)\in\mathcal{V}\}\\
  &= V\{(\mathbf L^-)^{n}:\mathcal{K}_{n}(\mathbf{L})\in \mathcal{V}\}\\ 
  &= V\{\mathbf{L}^-:\mathcal{K}_{n}(\mathbf L)\in \mathcal{V}\}\\
  &= \{\mathbf{L}^-:\mathcal{K}_{n}(\mathbf L)\in \mathcal{V}\},
\end{split}
\end{align}
where the last equality holds by Corollary~\ref{cor-HSPkites}.
Putting~\eqref{one}, \eqref{two} and~\eqref{three} together
we get the inclusion
$\psi(\mathcal{V})\subseteq \{\mathbf{L}^-:\mathcal{K}_{n}(\mathbf L)\in \mathcal{V}\}$
as required. 
\end{proof}

\begin{theorem}\label{PsiGamma}
For any $\mathcal{V}\in\Lambda^+(\var{\cigmv})$ and any
$n\in\mathbb{D}$, we have
$$
(\mathcal{V},n) = \Psi\Delta (\mathcal{V},n).
$$
\end{theorem}

\begin{proof}
Immediately from definitions we get
\begin{align*}
\Psi\Delta (\mathcal{V},n) &=\Psi (\delta (\mathcal{V})\cap \var{\ppmv}_{n})\\
                        &=\bigl(\psi(\delta (\mathcal{V})\cap
                              \var{\ppmv}_{n}),\ 
                              \ord(\delta (\mathcal{V})\cap \var{\ppmv}_{n})\bigr).
\end{align*}
By Lemma~\ref{psigamma} we have $\psi\delta
(\mathcal{V})=\mathcal{V}$, so we obtain the inclusion
$$
\psi(\delta (\mathcal{V})\cap \var{\ppmv}_{n})\subseteq \mathcal{V}.
$$
Conversely,  if $\mathbf A\in\mathcal{V}$ then $\mathcal{K}(\ell
(\mathbf A),1_A)\in \var{\ppmv}_1\subseteq \var{\ppmv}_n$ holds.  By the proof
of Lemma \ref{psigamma} we get that
$\mathcal{K}(\ell(\mathbf A),1_A)\in \delta(\mathcal{V})$ and hence
$$
\mathbf A=\mathbf{F}_{\mathcal{K}(\ell (\mathbf A),1_A)}\in
\psi(\delta (\mathcal{V})\cap \var{\ppmv}_{n})
$$
showing the other inclusion, so we have 
$\mathcal{V}=\psi(\delta (\mathcal{V})\cap \var{\ppmv}_{n})$
and therefore
$$
\Psi\Delta (\mathcal{V},n) =
\bigl(\mathcal{V},\ \ord(\delta (\mathcal{V})\cap\var{\ppmv}_{n})\bigr). 
$$

It remains to show that $\ord(\delta (\mathcal{V})\cap\var{\ppmv}_{n}) = n$.
As $\ord(\var{\ppmv}_{n}))=n$ by definition, 
$\ord(\mathbf{A})$ divides $n$ for any $\mathbf{A}\in \var{\ppmv}_{n}$.
Therefore $\ord(\delta (\mathcal{V})\cap \var{\ppmv}_{n})$ divides
$\ord(\var{\ppmv}_{n})$. To show that $n$ is attained,
take any nontrivial $\mathbf{A}\in\mathcal{V}$ (here we use the
assumption that $\mathcal{V}$ is nontrivial) and consider
$\mathcal{K}_n(\ell(\mathbf A))$. By construction
$\ord(\mathcal{K}_n(\ell (\mathbf A))) = n$ and it is readily 
checked that $\mathcal{K}_n(\ell (\mathbf{A}))\in\delta (\mathcal{V})$.
\end{proof}

\begin{lemma}\label{GammaPsi}
For any $\avar{V}\in\Lambda(\var{\ppmv})$ we have
$$
\avar{V}\subseteq \Delta\Psi (\avar{V}).
$$
\end{lemma}

\begin{proof}
$\avar{V}\subseteq \delta\psi(\avar{V})\cap \var{\ppmv}_{\ord(\avar{V})}=\Delta
(\psi (\avar{V}),\ord(\avar{V}))=\Delta\Psi (\avar{V})$.
\end{proof}

\begin{theorem}\label{ThmKiteGen}
Let $\avar{V}\in\Lambda (\var{\ppmv})$. The following are equivalent.
\begin{enumerate}
\item $\avar{V}$ is generated by kites.
\item $\avar{V}=\Delta\Psi (\avar{V})$.
\item $\avar{V}=\Delta (\mathcal{W},n)$ for some $\mathcal{W}\in\Lambda
  (\var{\cigmv})$ and some $n\in\mathbb{D}$.
\end{enumerate} 
\end{theorem}

\begin{proof}
First we prove that (2) and (3) are equivalent.  The implication
from (2) to (3) is immediate, since by definition $\Psi(\avar{V})$ is of the form
$(\avar{W},n)$ for some $\avar{W}\in\Lambda(\var{\cigmv})$ and some $n\in\mathbb{N}$. 
The converse is almost immediate, but we need to consider two cases
according to whether $\mathcal{W}$ is trivial or not.
If $\mathcal{W}=\var{Tr}$, then by Lemma~\ref{little-facts}(\ref{little-facts-4})
we have $\Delta(\var{Tr},n) = \var{BA}$ for any $n\in\mathbb{N}$, so
$$
\Delta\Psi(\var{BA})=\Delta(\var{Tr},1)=\var {BA}.
$$
If $\mathcal{W}$ is nontrivial, then Theorem~\ref{PsiGamma} applies,
giving  
$$
\Delta\Psi (\avar{V})= \Delta\Psi\Delta (\mathcal{W},n)=\Delta
(\mathcal{W},n)=\avar{V}.
$$
Thus, (3) implies (2). 

For the implication from (3) to (1), we will prove that
$\Delta(\mathcal{W},n)$ is generated by kites for any
$\mathbf{W}\in\Lambda(\var{\ppmv})$ and any $n\in\mathbb{N}$.
If $\mathcal{W} = \var{Tr}$, then
$\Delta(\mathcal{W},n) = \var{BA}$, and $\var{BA}$ is generated by $\bm{2}$,
which is a kite. So, assume $\mathcal{W}$ is nontrivial and fix an arbitrary
$n\in\mathbb{D}$. Since $\Delta(\mathcal{V},n)$ is
perfectly generated, it suffices to show that each perfect
$\mathbf{A}\in \Delta(\mathcal{V},n)$ is generated by kites.
Define a map
$f\colon \mathbf F_\mathbf A\rightarrow \mathbf F_{\mathcal K_n(\ell
  (\mathbf F_\mathbf A))}\ (=(\mathbf F_\mathbf A)^n)$
by $(f(x))(i)=x^{i\times\approx}$. Since $-^\approx$ is an
automorphism and $\ord(-^\approx)\mid n$ holds, $f$ is 
a homomorphism. Moreover, $f$ satisfies
$$
(f(x^\approx))(i)=x^{(i+1)\times\approx}=(f(x))(i+1)=(f(x)^\approx)(i)
$$
so Lemma~\ref{FilHom} applies. Using it we obtain a homomorphism
$$
h\colon \mathbf A\rightarrow
\mathcal K_n(\ell(\mathbf{F}_\mathbf{A}))
$$
which is injective, since $f$ is. 
It follows that $\mathbf A\in IS(\mathcal{K}_n(\ell (\mathbf F_\mathbf A)))$,
and so $\Delta (\mathcal{W},n)$ is generated by kites, as claimed.

Finally we will show the implication from (1) to (2).
Let $\avar{V}\in\Lambda (\var{\ppmv})$ be generated by kites, and let
$\ord(\mathcal{V}) = n$. Unwinding the definitions, we have  
$$
\Delta\Psi (\avar{V})=\Delta(\psi(\avar{V}), n)
=\delta\psi(\avar{V})\cap\var{\ppmv}_{n}.
$$
Now, as we already remarked, $\Psi(\avar{V})$ is always
of the form $(\avar{W},n)$, so by the previous part of the proof we
have that $\Delta\Psi(\avar{V})$ is generated by kites. 
Then, Lemma~\ref{cycleISP} shows that $\Delta\Psi(\avar{V})$ is
generated by kites of the form $\mathcal K_{\ord(\avar{V})}(\mathbf L)$.

Take any generator of $\Delta\Psi(\avar{V})$, or equivalently,
take any $\ell$-group $\mathbf{L}$ such that
$\mathcal K_{n}(\mathbf L)\in  \Delta\Psi (\avar{V})$.
Then, $\mathcal K_{n}(\mathbf L)\in  \delta\psi(\avar{V})$ and
so $\mathbf{F}_{\mathcal{K}_{n}(\mathbf{L})}\in \psi (\avar{V})$ by
Lemma~\ref{CharIden}. Since
$\mathbf{F}_{\mathcal{K}_{n}(\mathbf{L})} = (\mathbf{L}^-)^{n}$ and
$\psi(\avar{V})$ is a variety, we obtain 
$\mathbf L^-\in \psi (\avar{V})$, which yields
$\mathcal{K}_{n}(\mathbf{L})\in\avar{V}$, by
Lemma~\ref{kitegeneratedlemma}. This shows that all generators
of $\Delta\Psi(\avar{V})$ belong to $\avar{V}$, so
$\Delta\Psi (\avar{V})\subseteq \avar{V}$. The converse always holds
(cf. Lemma~\ref{GammaPsi}), so we obtain the desired equality
$\Delta\Psi (\avar{V}) = \avar{V}$.
\end{proof}

\begin{theorem}\label{ThmCompLatt}
Varieties generated by kites form a complete sublattice of
$\Lambda(\var{\ppmv})$ with $\var{\ppmv}$ being its largest, and $\var{BA}$
its smallest element.
\end{theorem}

\begin{proof}
Let $\Sigma(\var{\ppmv})$ be the subposet of $\Lambda(\var{\ppmv})$ formed by
the varieties generated by kites. Note that $\Sigma$ is an interior operator
on $\Lambda^+(\var{\ppmv})$, given by
$\Sigma(\mathcal{V}) = V\{\mathbf{A}\in\mathcal{V}: \mathbf{A} \text{ is a kite}\}$ 
for any $\mathcal{V}\in\Lambda^+(\var{\ppmv})$. Trivially, its fixed points
are precisely the members of $\Sigma(\var{\ppmv})$, so $\Sigma(\var{\ppmv})$ is
is a complete join subsemilattice of $\Lambda(\var{\ppmv})$, whose bottom
element is $\var{BA}$, and top element is $\var{\ppmv}$. 

On the other hand, by Theorem~\ref{ThmKiteGen} and lemmas leading up to it,
$\Delta\Psi$ is a closure operator on $\Lambda(\var{\ppmv})$.
Its closed elements are precisely the varieties generated by
kites, that is, members of $\Sigma(\var{\ppmv})$. Therefore, 
$\Sigma(\var{\ppmv})$ is a complete meet subsemilattice
of $\Lambda(\var{\ppmv})$, indeed of $\Lambda^+(\var{\ppmv})$.
Hence, $\Sigma(\var{\ppmv})$ is a complete join
subsemilattice of $\Lambda(\var{\ppmv})$, and so a complete sublattice of
$\Lambda(\var{\ppmv})$.
\end{proof}

It follows from the proof of Theorem~\ref{ThmCompLatt} that each perfectly
generated variety $\mathcal{V}$ of pseudo MV-algebras determines an interval
$[\Sigma(\avar{V}), \Delta\Psi(\avar{V})]$ in $\Lambda^+(\var{\ppmv})$.
If $\avar{V}$ is generated by kites, then 
$\Sigma(\avar{V}) = \avar{V} = \Delta\Psi(\avar{V})$, but
in general the interval is nontrivial, as we will now show.

\begin{lemma}\label{not-kite}
Let $\avar{V}$ be the variety generated by the algebra
$\mathcal{K}(\mathbb{Q}, \lambda)$ of Example~\ref{not-DvuKow}.
Then $\avar{V}$ is not generated by kites. Moreover,
$$
\bigl[\Sigma(\avar{V}), \Delta\Psi(\avar{V})\bigr] =
\bigl[\var{BA}, V\{\mathbf{A}\in\var{\pfpmv}: \mathbf{F}_\mathbf{A}
\text{ is commutative}\}\bigr].
$$ 
\end{lemma}

\begin{proof}
By Theorem~5.13 of Dvure\v{c}enskij, Kowalski~\cite{DK14},
every kite is subdirectly embeddable in a
product of subdirectly irreducible kites, so $\Sigma(\avar{V})$
is generated by subdirectly irreducible kites from $\avar{V}$.
Let $\mathbf{K}\in\avar{V}$ be a subdirectly irreducible kite. 
By J\'onsson's Lemma $\mathbf{K}\in HSP_U(\mathcal{K}(\mathbb{Q},\lambda))$
and therefore $\mathbf{K}$ is linearly ordered, because
$\mathcal{K}(\mathbb{Q},\lambda)$ is.  Therefore,
$\mathbf{K} = \mathcal{K}(\mathbf{L}^1, \tau)$, where $\mathbf{L}$
is an Abelian linearly ordered $\ell$-group, and
$\tau\colon 1\rightarrow 1$ is the constant map.
Hence $\Sigma(\avar{V})\models x\sn\sn = x$, and so
$\mathcal{K}(\mathbb{Q},\lambda)\notin \Sigma(\avar{V})$. 
This proves the first part.

For the second part, a simple calculation shows that $\ord(\lambda) =
0$, and so we have $\ord(\avar{V}) = \ord(\mathcal{K}(\mathbb{Q}, \lambda)) = 0$.
Since $\mathbb{Q}$ generates the variety of Abelian $\ell$-groups, we get
that $\mathbb{Q}^-$ generates the commutative subvariety of $\var{\cigmv}$.
Therefore 
$$
\Delta(V(\mathcal{K}(\mathbb{Q}, \lambda)), 0) =
V\{\mathbf{A}\in\var{\pfpmv}: \mathbf{F}_\mathbf{A}\text{ is commutative}\}.
$$
On the other hand, $\mathcal{K}(\mathbb{Q}, \lambda)$ satisfies the identity
$$
(x\vee x^-)\odot(x\vee x^-) = (x\vee x^-)\sn\sn.
$$
To see it, note that $a\vee a^-\in F_{\mathcal{K}(\mathbb{Q}, \lambda)}$ for any
$a\in \mathcal{K}(\mathbb{Q}, \lambda)$, and so taking $b = a\vee a^-$ we get
$b\odot b = 2b = b\sn\sn$. By the proof of the first part of the theorem,
we get that every kite in $\avar{V}$ satisfies $x\sn\sn = x$.
Therefore
$$
\Sigma(\avar{V})\models (x\vee x^-)\odot(x\vee x^-) = (x\vee x^-)\sn\sn = x\vee x^-
$$
and it follows by a direct calculation that every kite in
$\Sigma(\avar{V})$ must be idempotent. The only idempotent kite is the
two-element Boolean algebra.   
\end{proof}

We end by a more detailed description of the lattice of
varieties generated by kites. Incidentally, it answers Questions~8.1 and~8.2
from Dvure\v{c}enskij, Kowalski~\cite{DK14} insofar as they apply in this context.

\begin{theorem}\label{Varieties}
Let $\mathbb{K}$ be the lattice of subvarieties of $\var{\ppmv}$ generated by kites.  
$$
\mathbb{K}\cong
\mathbf{1}\oplus\bigl(\Lambda^+(\var {\cigmv})\times\mathbb{D}\bigr)
\cong
\mathbf{1}\oplus\bigl(\Lambda^+(\var{LG})\times\mathbb{D}\bigr)
$$
where $\mathbf{1}$ is the trivial lattice and $\oplus$ is the operation of
ordinal sum. 
\end{theorem}

\begin{proof}
Consider the map $\Delta\colon \Lambda(\var
{\cigmv})\times\mathbb{D} \rightarrow \mathbb{K}$
of Definition~\ref{Galois-maps}. This map is surjective, by
Theorem~\ref{ThmKiteGen}(2). If $\avar{V}_1$ and $\avar{V}_2$
are distinct nontrivial varieties in $\var{\cigmv}$, then
$\delta(\avar{V}_1)\neq \delta(\avar{V}_2)$ by Lemma~\ref{CharIden}.
Together with the fact that $\var{\ppmv}_n\neq \var{\ppmv}_m$
for $n\neq m$, it implies that $\Delta$ is injective on 
$\Lambda^+(\var{\cigmv})\times\mathbb{D}$.

Let $\Delta'\colon \mathbf{1}\oplus\bigl(\Lambda^+(\var
{\cigmv})\times\mathbb{D}\bigr) \rightarrow \mathbb{K}$
be defined by putting $\Delta'(\mathbf{1}) = \var{BA}$ and
$\Delta'(\avar{V},n) = \Delta(\avar{V},n)$ for
$\avar{V}\in\Lambda^+(\var {\cigmv})\times\mathbb{D}$.
By Lemma~\ref{little-facts} and the previous paragraph, $\Delta'$ is
bijective and monotone. Note that $\mathbf{1}$ and $\var{BA}$ are the respective
bottom elements of $\mathbf{1}\oplus\bigl(\Lambda^+(\var
{\cigmv})\times\mathbb{D}\bigr)$ and $\mathbb{K}$), so
$\Delta'$ also preserves and reflect the bottom element.

Now, take $\avar{W}_1,\avar{W}_2\in\mathbb{K}$ such that
$\var{BA}\subset \avar{W}_1\subseteq\avar{W}_2$.
By Theorem~\ref{ThmKiteGen}(3) we have $\avar{W}_1 = \Delta(\avar{V}_1, n_1)$ and
$\avar{W}_2 = \Delta(\avar{V}_2, n_2)$, for some
$\avar{V}_1, \avar{V}_2\in\Lambda^+(\var{\cigmv})$,
and $n_1,n_2\in\mathbb{D}$. Hence,
$\delta(\avar{V}_1) = \avar{W}_1\subseteq\avar{W}_2 = \delta(\avar{V}_2)$,
and $n_1 = \ord(\avar{W}_1) \mid \ord(\avar{W}_2) = n_2$.
It follows that $\{\mathbf A\in
\var {\pfpmv}: \mathbf{F}_{\mathbf{A}}\in\mathcal{V}_1\}\subseteq
\{\mathbf A\in
\var {\pfpmv}: \mathbf{F}_{\mathbf{A}}\in\mathcal{V}_2\}$, which further implies
$\avar{V}_1\subseteq \avar{V}_2$. So,
$(\avar{V}_1,n_1)\leq(\avar{V}_2,n_2)$ and therefore $\Delta'$ reflects the
order.

A bijective map  preserving and reflecting order between lattices is a lattice
isomorphism, so $\Delta'$ is the required map between $\mathbb{K}$ and 
$\mathbf{1}\oplus\bigl(\Lambda^+(\var {\cigmv})\times\mathbb{D}\bigr)$.

The second isomorphism follows by the well known fact
that $\Lambda(\var{\cigmv})\cong\Lambda(\var{LG})$ (see the remarks preceding Proposition~\ref{aut-commut}).
\end{proof}  

\section{Perspectives}

We envisage several directions for further work. Below we briefly describe two.

\subsection{Perfect FL-algebras}
An antivariety is a class of similar algebras (or, more generally, 
relational structures) defined by a set of
\emph{anti-identities}, that is, universally quantified disjunctions of negated
identities (atomic formulas). Equivalently, an antivariety is a class
closed under inverses of homomorphic images.
Since its introduction in~\cite{GK00},
the theory of antivarieties has been growing steadily,
especially in areas connected to constraint satisfaction.\footnote{Recall that
$H^{-1}(\mathbf{A})$ for a single structure $\mathbf{A}$ is
precisely the constraint satisfaction problem with template $\mathbf{A}$.}

Definition~\ref{perfect} naturally extends to all FL-algebras: we
say an algebra $\mathbf{A}\in \var{FL}$ is perfect if
there exists a homomorphism $h_\mathbf A\colon \mathbf A\rightarrow\bm 2$ such
that
\begin{itemize}
\item[($\ast$)] for any $x\in h_\mathbf A^{-1}(0)$ and any $y\in h_\mathbf A^{-1}(1)$ the
inequality $x\leq y$ holds.
\end{itemize}
Without ($\ast$), we have precisely the antivariety
$H^{-1}(\mathbf{2})$ of all FL-algebras that have $\mathbf{2}$ as a homomorphic
image. But ($\ast$) is very natural model theoretically, namely, it requires
that the homomorphism  $h_\mathbf A\colon \mathbf A\rightarrow\bm 2$ be
\emph{strong} with respect to the lattice ordering, that is, not only order
preserving, but also order reflecting. The latter makes the class of perfect
FL-algebras not closed under direct products. By Lemma~\ref{HSPu-closure}
and the remarks following it, the class is closed under nontrivial homomorphic
images, and perfect products. Indeed it is not difficult to show that
perfect FL-algebras are the largest subclass of $H^{-1}(\mathbf{2})$
with these properties. 

We intend to investigate perfect FL-algebras further: algebraically and
model-theoretically.

\subsection{Generalised kites in general}

Kites were conceived in the context of pseudo BL-algebras, and it would be
natural to try to lift our results to that setting, and possibly beyond.
Promisingly, Definition~\ref{gen-kite} extends:
instead of an $\ell$-group $\mathbf{L}$ and an
automorphism $\lambda$ one needs to take two $\ell$-groups
$\mathbf{L}$ and $\mathbf{M}$ and two pairs of adjoint maps between
them. To be more precise, one takes the lattice-ordered monoid reducts of
$\mathbf{L}^-$ and $\mathbf{M}^-$, and pairs $(\lambda_*,\lambda^*)$ and
$(\rho_*,\rho^*)$, such that
\begin{enumerate}
\item[(i)] $\lambda_*, \rho_*\colon \mathbf{L}^- \rightarrow
\mathbf{M}^-$  and $\lambda^*, \rho^*\colon \mathbf{M}^- \rightarrow \mathbf{L}^-$,
\item[(ii)] $(\lambda_*,\lambda^*)$ and $(\rho_*,\rho^*)$ are adjoint pairs with
respect to $\leq$,
\item[(iii)] $\lambda_*$ and $\rho_*$ preserve multiplication and unit.
\end{enumerate}  
Then, an obvious modification of Definition~\ref{gen-kite} works, producing
a perfect FL$_w$-algebra. If $\lambda_*\circ\lambda^*$ and $\rho_*\circ\rho^*$
are both the identity on $M^-$, then the resulting algebra is a perfect
pseudo BL-algebra. Moreover, if $\lambda_*$ is an automorphism, then
$\lambda^* = (\lambda_*)^{-1}$; if furthermore $\rho^* = \rho_* = id$, then we
obtain exactly Definition~\ref{gen-kite}.
We intend to investigate this construction further. 

\section{Acknowledgements}

We are grateful to an anonymous referee for very thorough reading and
suggestions that led to substantial improvements in presentation.
We are also grateful to Anatolij Dvure\v{c}enskij for pointing out
in private communication that our Theorem~\ref{cat-eqv-ppMV-LGA} generalises
a result in Di Nola, Dvure\v{c}enskij, Tsinakis~\cite{DDT08}.

\end{document}